\documentclass[11pt]{amsart}

\usepackage{amsrefs}

\usepackage{epsfig}  		
\usepackage{epic,eepic}       
\usepackage{graphicx}

\usepackage{enumerate}
\usepackage{mathbbol}
\usepackage{amssymb}
\usepackage{braket}
\usepackage{longtable}
\usepackage[colorlinks]{hyperref}
\usepackage{caption}
\usepackage{mathtools}
\usepackage{microtype}
\usepackage{soul,xcolor}
	\setstcolor{red}

\newenvironment{claimproof}[1][Proof of Claim]{\begin{proof}[#1]}{\end{proof}}

\newtheorem{lemma}{Lemma}[section]
\newtheorem*{lemma*}{Lemma}
\newtheorem{theorem}[lemma]{Theorem}
\newtheorem*{theorem*}{Theorem}

\newtheorem{proposition}[lemma]{Proposition}

\newtheorem*{proposition*}{Proposition}
\newtheorem{fact}[lemma]{Fact}
\newtheorem*{fact*}{Fact}

\newtheorem*{notation*}{Notation}
\newtheorem*{conventions*}{Conventions}
\newtheorem{remark}[lemma]{Remark}
\newtheorem*{remark*}{Remark}

\newtheorem*{corollary*}{Corollary}
\newtheorem{conjecture}{Conjecture}
\newtheorem*{conjecture*}{Conjecture}

\newtheorem*{problem*}{Problem}

\newtheorem*{question*}{Question}

\newtheorem{assumption*}{Assumption}

\theoremstyle{definition}
\newtheorem{example}{Example}
\newtheorem*{example*}{Example}
\newtheorem{definition}[lemma]{Definition}
\newtheorem*{definition*}{Definition}

\theoremstyle{remark}
\newtheorem{claim}{Claim}
\newtheorem*{claim*}{Claim}

\newtheorem*{case*}{Case}
\newtheorem*{construction*}{Construction}

\newtheorem*{exercise*}{Exercise}

\numberwithin{equation}{section}

\newcommand{\Z}{\mathbb{Z}}
\newcommand{\R}{\mathbb{R}}  
\newcommand{\C}{\mathbb{C}}
\newcommand{\Q}{\mathbb{Q}}
\newcommand{\I}{\mathcal{I}}

\newcommand{\U}{\mathcal{U}}

\newcommand{\bs}{\backslash}

\renewcommand\AA{{\mathcal A}}

\newcommand\CC{{\mathcal C}}

\newcommand\II{{\mathcal I}}
\newcommand\JJ{{\mathcal J}}

\newcommand\<{\langle}
\renewcommand\>{\rangle}
\newcommand\az{{\aleph_0}}

\newcommand{\tp}{\mathrm{tp}}
\newcommand{\qftp}{\mathrm{qftp}}

\newcommand{\Wr}{\mathrm{\thinspace Wr \thinspace}} 

\newcommand{\ra}{\rightarrow}
\newcommand{\Ra}{\Rightarrow}

\def\abar{\bar{a}}

\def\bbar{\bar{b}}
\def\cbar{\bar{c}}
\def\dbar{\bar{d}}
\def\ebar{\bar{e}}

\def\hbar{\bar{h}}

\def\mbar{\bar{m}}
\def\nbar{\bar{n}}

\def\xbar{\bar{x}}
\def\ybar{\bar{y}}
\def\zbar{\bar{z}}

\def\<{\langle}
\def\>{\rangle}

\renewcommand{\C}{\mathfrak{C}}
\title{Characterizations of monadic NIP}
\author{Samuel Braunfeld}
\author{Michael C. Laskowski$^*$}
\thanks{$^*$Partially supported
	by NSF grant DMS-1855789}

\def\rtp{{\rm rtp}}

\def\qftp{{\rm qftp}}

\def\P{{\mathcal P}}
\def\J{{\mathcal J}}

\newcommand{\joined}{joined tuple-coding configuration}

\begin{document}

	\begin{abstract}
		We give several characterizations of when a complete first-order theory $T$ is monadically NIP, i.e. when expansions of $T$ by arbitrary unary predicates do not have the independence property. The central characterization is a condition on finite satisfiability of types. Other characterizations include decompositions of models, the behavior of indiscernibles, and a forbidden configuration. As an application, we prove non-structure results for  hereditary classes of finite substructures of non-monadically NIP models that eliminate quantifiers.
	\end{abstract}
	
	\maketitle
	
		\textcolor{red}{An appendix has been added, containing two corrections. The first involves replacing indiscernible-triviality with {\em endless indiscernible triviality} (both defined in Definition \ref{def:ind triv}) in Theorem \ref{thm:main}. The second withdraws the claimed proof that $Age(T)$ is not 4-wqo in Theorem \ref{thm:growth}. We have added red text and strikethroughs noting the affected results.}
	
	\section{Introduction}

	It is well known that many first order theories whose models are tame can become unwieldy after naming a unary predicate.  Arguably the best known example of this is the field $(\C,+,\cdot)$ of complex numbers.  Its theory is uncountably categorical, but after naming a predicate for the real numbers, the expansion becomes unstable.  A more extreme example is the theory $T$ of infinite dimensional vector spaces over a finite field, in a relational language.  The theory $T$ is totally categorical, but if, in some model $V$, one names a basis $B$, then by choosing specified sum sets of basis elements, one can code arbitrary bipartite graphs in expansions of $V$ by unary predicates.
	
	As part of a larger project in \cite{BS}, Baldwin and Shelah undertook a study of this phenomenon.    They found that a  primary dividing line is whether $T$ {\em admits coding} i.e., there are three subsets $A,B,C$ of a model of $T$ and a formula $\phi(x,y,z)$ that defines a pairing function $A\times B\rightarrow C$.   
	If one can find such a configuration in a model $M$ of $T$, some monadic expansions of $M$ are wild.  The primary focus in \cite{BS} was
	{\em monadically stable} theories, i.e. theories that remain stable after arbitrary expansions by unary predicates.   Clearly, the two theories  described above are  stable, but not monadically stable.  
	They offered a  characterization of monadically stable theories within the stable theories via a condition on the behavior of non-forking. This allowed them to prove that monadic stability yields a dividing line within stable theories: models of monadically stable theories are well-structured and admit a nice decomposition into trees of submodels, while if a theory is stable but not monadically stable then it encodes arbitrary bipartite graphs in a unary expansion,
	and so is not even monadically NIP.
	
	A theory  $T$ is NIP if it does not have the independence property, and is monadically NIP if every expansion of a model of $T$ by unary predicates is also NIP.    The behavior of NIP theories has been extensively studied, see e.g., \cite{Pierre}.
	Soon after \cite{BS}, Shelah further studied monadically NIP theories in \cite{Hanf}, where he showed they satisfy a condition on the behavior of finite satisfiable types paralleling the condition on the behavior of non-forking in monadically stable theories. He was then able to use this to produce a linear decomposition of models of monadically NIP theories, akin to a single step of the tree decomposition in monadically stable theories.
	
	We dub Shelah's condition on the behavior of finite satisfiability the {\em f.s. dichotomy}, and we consider it to be the  fundamental property expressible in the original language $L$ describing  the dichotomous behavior  outlined above. We show the f.s. dichotomy characterizes monadically NIP theories and provide several other characterizations, including admitting a linear decomposition in the style of Shelah, a forbidden configuration, and conditions on the behavior of indiscernible sequences after adding parameters. Definitions for the following theorem may be found in Definitions \ref{def:tuple code}, \ref{def:fs dich}, \ref{def:decomp}, and \ref{def:ind triv}. Of note is that all but the first two conditions refer to the theory $T$ itself, rather than unary expansions.
	
	\textcolor{red}{In Theorem \ref{thm:main}, ``indiscernible-triviality'' has been replaced with ``endless indiscernible triviality''.}
	
	\begin{theorem} \label{thm:main}   The following are equivalent for a complete theory $T$ with an infinite model.
		\begin{enumerate}
			\item  $T$ is monadically NIP.
			\item  No monadic expansion of $T$ admits coding.
			\item  $T$ does not admit coding on tuples.
			\item  $T$ has the f.s.\ dichotomy.
			\item  For all $M^* \models T$ and $M, N \preceq M^*$, every partial $M$-f.s.\ decomposition of $N$ extends to an (irreducible) $M$-f.s.\ decomposition of $N$.
			\item  $T$ is dp-minimal and has \textcolor{red}{endless} indiscernible triviality.
		\end{enumerate}
	\end{theorem}
	
	We believe that monadic NIP (or perhaps a quantifier-free version) is an important dividing line in the combinatorics of hereditary classes, and provides a general setting for the sort of decomposition arguments common in structural graph theory. For example, see the recent work on bounded twin-width  in the ordered binary case, where it coincides with monadic NIP \cites{ST, TW4}. Here, we mention the following conjecture, adding monadic NIP to a question of Macpherson \cite{MacHom}*{Question 2.2.7}.
	\begin{conjecture} \label{conj:hom}
		The following are equivalent for a countable homogeneous $\omega$-categorical relational structure $M$.
		\begin{enumerate}
			\item $M$ is monadically NIP.
			\item The (unlabeled) growth rate of $Age(M)$ is at most exponential.
			\item $Age(M)$ is well-quasi-ordered under embeddability, i.e. it has no infinite antichain.
		\end{enumerate}
	\end{conjecture} 
	
	From Theorem \ref{thm:main}, we see that if $T$ is not monadically NIP then it admits coding on tuples. This allows us to prove the following non-structure theorem in \S \ref{s:fin} (with Definition \ref{def:fin} defining the relevant terms), in particular confirming $(2) \Rightarrow (1)$ and a weak form of $(3) \Rightarrow (1)$ from the conjecture, although without any assumption of $\omega$-categoricity.
	
		\textcolor{red}{In Theorem \ref{thm:growth},  the claim regarding 4-wqo remains unproven.}
	
	\begin{theorem} \label{thm:growth}
		Suppose $T$ is a complete theory with quantifier elimination  in a relational language with finitely many constants.
		If $T$ is not monadically NIP, then  $Age(T)$ has growth rate asymptotically greater than $(n/k)!$ for some $k \in \omega$ \st{and is not 4-wqo}.
	\end{theorem}
	
	We also show the following, partially explaining the importance of monadic model-theoretic properties for the study of hereditary classes.
	
	\begin{theorem} \label{thm:icoll}
		Suppose $T$ is a complete theory with quantifier elimination  in a relational language with finitely many constants.
		Then $Age(T)$ is NIP if and only if $T$ is monadically NIP, and $Age(T)$ is stable if and only if $T$ is monadically stable.
	\end{theorem}
	
	In Section 2, we review basic facts about finite satisfiability, and introduce {\em $M$-f.s. sequences}, which are closely related to, but more general than Morley sequences. The results of this section apply to an arbitrary theory, and so may well be of interest beyond monadic NIP. Section 3 introduces the f.s. dichotomy and proves the equivalence of (3)-(6) from Theorem \ref{thm:main}. Much of these two sections is an elaboration on the terse presentation of \cite{Hanf}, although there are new definitions and results, particularly in Section 3.2, which deals with the behavior of indiscernibles in monadically NIP theories. In Section 4 we finish proving the main theorem by giving a type-counting argument that the f.s. dichotomy implies monadic NIP, and by showing that if $T$ admits coding on tuples then it admits coding in a unary expansion. In Section 5, we prove Theorems \ref{thm:growth} and \ref{thm:icoll}.
	
	We are grateful to Pierre Simon, with whom we have had numerous insightful discussions about this material.  In particular, the relationship between monadic NIP and indiscernible-triviality was suggested to us by him.

	\subsection{Notation}
	Throughout this paper, we work in $\C$, a large, sufficiently saturated, sufficiently homogeneous model of a complete theory $T$.  
	We routinely consider $\tp(A/B)$ when $A$ is an infinite set.  To make this notion precise, we (silently) fix an enumeration $\abar$ of $A$ (of ordinal order type)
	and an enumeration $\xbar$ with $\lg(\xbar)=\lg(\abar)$.  Then $\tp(A/B)=\{\theta(\xbar',\bbar):\C\models\theta(\abar',\cbar)$  for all subsequences $\xbar'\subseteq\xbar$ and 
	$\abar'$ is the corresponding subsequence of $\abar\}$.

	\section{$M$-f.s.\ sequences}
	
	Forking independence and Morley sequences are fundamental tools in the analysis of monadically stable theories in \cite{BS}. These are less well-behaved outside the stable setting, but in any theory we may view `$\tp(A/MB)$ is finitely satisfiable in $M$' as a statement that $A$ is (asymmetrically) independent from $B$ over $M$. Following \cite{Hanf}, we will use finite satisfiability in place of non-forking, and indiscernible $M$-.f.s. sequences in place of Morley sequences. Throughout Section~2, we make no assumptions about the complexity of $Th(\C)$.
	
	\subsection{Preliminary facts about $M$-f.s.\ sequences}
	
	For the whole of this section, fix a small $M\preceq\C$ (typically, $|M|=|T|$). 
	
	\begin{definition}  Suppose $B\supseteq M$.  Then for any $A$ (possibly infinite) we say $\tp(A/B)$ is {\em finitely satisfied in $M$} if, for all $\theta(\ybar,\bbar)\in\tp(A/B)$,
		there is $\mbar\in M^{\lg(\ybar)}$ such that $\C\models\theta(\mbar,\bbar)$.
	\end{definition}
	
	One way of producing finitely satisfiable types in $M$ comes from {\em average types.}
	
	\begin{definition}
		Suppose $\xbar$ is a possibly infinite tuple.  For any ultrafilter $\U$ on $M^{\lg(\xbar)}$ and any $B\supseteq M$,
		$$Av(\U,B)=\set{\phi(\xbar,\bbar):\set{\mbar\in M^{\lg(\xbar)}:\C\vDash\phi(\mbar,\bbar)}\in\U}$$
	\end{definition}
	
	It is easily checked that $Av(\U,B)$ is a complete type over $B$ that is finitely satisfied in $M$.  We record a few basic facts about types that are finitely satisfied in $M$.
	Proofs can be found in either Section VII.4 of \cite{Shc} or in \cite{Pierre}.
	
	\begin{fact} \label{basic}   Let $M$ be any model.
		\begin{enumerate}
			\item  For any set $B\supseteq M$ and any $p(\xbar)\in S(B)$ ($\xbar$ may be an infinite tuple),
			$p$ is finitely satisfied in $M$ if and only if $p=Av(\U,B)$ for some ultrafilter $\U$ on $M^{\lg(\xbar)}$.
			\item  Suppose $\Gamma(\xbar,B)$ is any set of formulas, closed under finite conjunctions, and each of which is realized in $M$.
			Then
			there is a complete type $p\in S(B)$ extending $\Gamma$ that is finitely satisfied in $M$.
			\item  (Non-splitting)  If $p\in S(B)$ is finitely satisfied in $M$, then $p$ does not split over $M$, i.e., if $\bbar,\bbar'\subseteq B$ and $\tp(\bbar/M)=\tp(\bbar'/M)$,
			then for any $\phi(\xbar,\ybar)$, $\phi(\xbar,\bbar)\in p$ if and only if $\phi(\xbar,\bbar')\in p$.
			\item  (Transitivity)  If $\tp(B/C)$ and $\tp(A/BC)$ are both finitely satisfied in $M$, then so is $\tp(AB/C)$.
		\end{enumerate}

	\end{fact}
	\begin{definition} [$M$-f.s. sequence] \label{AJ} With $M$ fixed as above, let  $(I,\le)$ be any linearly ordered index set.
		\begin{itemize}
			\item  Suppose $\<A_i:i\in I\>$ is any sequence of sets, indexed by $(I,\le)$.  For $J\subseteq I$, $A_J$ denotes $\bigcup_{j\in J} A_j$, and for $i^*\in I$,
			$A_{<i^*}$ denotes $\bigcup_{i<i^*} A_i$.  $A_{\le i^*}$ and $A_{>i^*}$ are defined analogously.
			\item  For $C\supseteq M$, an {\em $M$-f.s.\ sequence over $C$},  is a  sequence of sets $\<A_i:i\in I\>$ such that $\tp(A_i/A_{<i}C)$ is finitely satisfied in $M$ for every $i\in I$.
			When $C=M$ we simply say $\<A_i:i\in I\>$ is an $M$-f.s.\ sequence.
		\end{itemize}
	\end{definition}
	
	Note that for any $C\supseteq M$, $\<A_i:i\in I\>$ is an $M$-f.s.\ sequence over $C$ if and only if the concatenation
	$\<C\>\smallfrown\<A_i:i\in I\>$ is an $M$-f.s.\ sequence.
	
	We note two useful operations on $M$-f.s.\ sequences over $C$, `Shrinking' and `Condensation'.
	
	\begin{definition}
		Suppose $C\supseteq M$ and $\<A_i:i\in I\>$ is an $M$-f.s.\ sequence over $C$.
		\begin{enumerate}
			\item `Shrinking:'  For every $J\subseteq I$, for all $A_j'\subseteq A_j$, and for all $C'$ with $C\supseteq C'\supseteq M$, we say
			$\<A_j':j\in J\>$ as a sequence over $C'$ is obtained by {\em shrinking} from $\<A_i:i\in I\>$ as a sequence over $C$.
			\item Condensation:'  Suppose $\pi:I\rightarrow J$ is a condensation, i.e., a surjective map with each $\pi^{-1}(j)$ a convex subset of $I$.
			For each $j\in J$, let $A_j^*:=\bigcup\set{A_i:i\in \pi^{-1}(j)}$. We say $\<A_j^*:j\in J\>$ as a sequence over $C$ is obtained by {\em condensation} from $\<A_i:i\in I\>$ as a sequence over $C$.
		\end{enumerate}
	\end{definition}
	
	In particular, removing a set of $A_i$'s from the sequence is an instance of Shrinking.
	
	\begin{lemma}  \label{manipulations}  Suppose $C\supseteq M$ and $\<A_i:i\in I\>$ is an $M$-f.s.\ sequence over $C$. Then Shrinking and Condensation both preserve being an $M$-f.s. sequence over $C$. In particular, for any partition $I=J\sqcup K$ into convex pieces, the two-element sequence $\<A_J,A_K\>$ is an $M$-f.s.\ sequence over $C$.
	\end{lemma}
	
	\begin{proof}  The statement is immediate for Shrinking, and for Condensation follows by transitivity in Fact \ref{basic}. 
		The last sentence is a special case of Condensation, as the partition defines a condensation $\pi:I\rightarrow \set{0,1}$ with $\pi^{-1}(0)=J$.
	\end{proof}
	
	\begin{definition} \label{ext}  If $\<A_i:i\in I\>$ is an $M$-f.s.\ sequence over $C$, call $\<B_j:j\in J\>$ a {\em simple extension, resp.\ blow-up} if $\<A_i:i\in I\>$ is attained from it by Shrinking,
		resp.\ by Condensation.  $\<D_k:k\in K\>$ is an {\em extension} of $\<A_i:i\in I\>$ if it is a blow-up of a simple extension of $\<A_i:i\in I\>$ over $C$.
	\end{definition}
	
	Here is one general result, whose verification is just bookkeeping.
	
	\begin{lemma}  \label{blow} Suppose $\<A_i:i\in I\>$ is an $M$-f.s.\ sequence, $i^*\in I$, $J \cap I = \emptyset$,  and $\<A_j':j\in J\>/MA_{<i^*}$ is an $M$-f.s.\ sequence over $MA_{<i^*}$ with
		$\bigcup\set{A_j':j\in J}=A_{i^*}$.  Then the blow-up
		$$\<A_i:i<i^*\>\smallfrown \<A_j':j\in J\>\smallfrown \<A_i:i>i^*\>$$
		is also an $M$-f.s.\ sequence.
	\end{lemma}
	
	The next lemma is not used later, but shows that if $M \preceq N$, then decomposing $N$ as an $M$-f.s. sequence gives a chain of elementary substructures approximating $N$.
	
	\begin{lemma}  \label{chainofmodels}
		Suppose $M\preceq N$ and $\<A_i:i\in I\>$ is any $M$-f.s.\ sequence with $MA_I=N$.  Then, for every initial segment $I_0\subseteq I$ (regardless of whether or not $I_0$ has a maximum) $MA_{I_0}$ is an elementary substructure of $N$.
	\end{lemma}  
	
	\begin{proof}  We apply the Tarski-Vaught criterion.  Choose a formula $\phi(x,\abar,\mbar)$ with $\abar$ from $A_{I_0}$ and $\mbar$ from $M$ such that $N\models\exists x\phi(x,\abar,\mbar)$. If some $c\in N\setminus MA_{I_0}$ realizes $\phi(x,\abar,\mbar)$, then as $\tp(c/MA_{I_0})$ is finitely satisfied in $M$, there is also a solution in $M$. Otherwise, if there is a solution in $MA_{I_0}$, there is nothing to check. 
	\end{proof}

	The following argument is contained in the proof of \cite{Hanf}*{Part I Lemma 2.6}, but the statement here is slightly more general. (The paper \cite{Hanf} is divided into Part I and Part II, with overlapping numbering schemes.) 

	\begin{proposition}[Extending the base]  \label{baseextension} Suppose $C\supseteq M$ and $\<A_i:i\in I\>$ is an $M$-f.s.\ sequence over $C$.
		For every $D\supseteq C$, there is $D'$ with $\tp(D'/C)=\tp(D/C)$ and $\<A_i:i\in I\>$ is an $M$-f.s.\ sequence over $D'$.  
	\end{proposition}
	
	\begin{proof}  As notation, choose disjoint sets $\set{\xbar_i:i\in I}$ of variables, with $\lg(\xbar_i)=\lg(A_i)$ for each $i\in I$.  
		For each $i\in I$, choose an ultrafilter $\U_i$ on $M^{\lg(A_i)}$ such that $\tp(A_i/A_{<i}C)=Av(\U_i,A_{<i}C)$.
		
		For a finite, non-empty $t=\set{i_1<i_2<\dots<i_n}\subseteq I$, let $\xbar_t=\xbar_{i_1}\dots\xbar_{i_n}$.
		We will recursively define complete types $w_t(\xbar_t)\in S_{\xbar_t}(D)$ as follows:
		\begin{itemize}
			\item  For $t=\set{i^*}$ a singleton, let $w_t(\xbar_t):=Av(\U_{i^*},D)$.
			\item  For $|t|>1$, letting $i^*=\max(t)$ and $s=t\setminus \set{i^*}$,
			$$w_t(\xbar_t):=w_s(\xbar_s)\cup Av(\U_i,D\xbar_s)$$
		\end{itemize}
		That is, $\abar_t'$ realizes $w_t$ if and only if $\abar'_s$ realizes $w_s$ and, for every $\theta(\xbar_{i^*},\dbar,\abar_s')$,
		$\theta(\abar'_{i^*},\dbar,\abar'_s)$ holds if and only if $\set{\mbar\in M^{\lg(A_{i^*})}:\C\vDash\theta(\mbar,\dbar,\abar_s')}\in\U_{i^*}$.
		
		It is easily checked that each $w_t(\xbar_t)$ is a complete type over $D$ and, arguing by induction on $|t|$, whenever $t'\subseteq t$,
		$w_{t'}$ is the restriction of $w_t$ to $\xbar_{t'}$.
		Thus, by compactness, $w^*:=\bigcup\set{w_t(\xbar_t):t\subseteq I \text{ non-empty, finite}}$ is  consistent, and in fact, is a complete type over $D$.
		Choose any realization $\<A_i':i\in I\>$ of $w^*$.
		Then, for each $i\in I$, $\tp(A_i'/DA_{<i}')=Av(\U_i,DA_{<i}')$.  Since $D\supseteq C$ and $\tp(A_i/CA_{<i})=Av(\U_i,CA_{<i})$, it follows
		that $\tp(\<A_i':i\in I\>/C)=\tp(\<A_i:i\in I\>/C)$.  Thus, it suffices to choose any $D'$ satisfying
		$\<A_i':i\in I\>D\equiv_{C} \<A_i:i\in I\>D'$.
	\end{proof}
	
	\subsection{$C\supseteq M$ full for non-splitting}
	
	\begin{definition}
		We call $C\supseteq M$ {\em full} (for non-splitting over $M$) if, for every $n$, every $p\in S_n(M)$ is realized in $C$.
	\end{definition}
	
	The relevance of fullness is that, whenever $C\supseteq M$ is full,  every complete type $q\in S(C)$ has a unique extension to any set $D\supseteq C$ that does not split over $M$.
	Keeping in mind finite satisfiability as an analogue of non-forking, the next lemma says that `types over $C$ that are finitely satisfied in $M$ are stationary.'
	
	\begin{lemma} [\cite{Hanf}*{Part I Lemma 1.5}] \label{stationary}  Suppose $C\supseteq M$ is full and $p\in S(C)$ is finitely satisfied in $M$.  Then for any set $D\supseteq C$,
		there is a unique $q\in S(D)$ extending $p$ that remains finitely satisfied over $M$.
	\end{lemma}
	
	\begin{proof}  In fact, we can easily describe $q$.  A formula $\theta(\xbar,\dbar)\in q$ if and only if $\theta(\xbar,\cbar)\in p$ for some
		(equivalently, for every) $\cbar$ from $C$ with $\tp(\dbar/M)=\tp(\cbar/M)$.  The fact that $q$ is well-defined is because, being finitely satisfied in $M$,
		$p$ does not split over $M$.
	\end{proof}
	
	\begin{lemma}  [\cite{Hanf}*{Part I Observation 1.6}] \label{1.6} Suppose $C\supseteq M$ is full and $\<A,B\>/C$ is an $M$-f.s.\ sequence over $C$.  Partition $B$ as $B_1B_2$ (not necessarily convex).
		Then $\<A,B_1,B_2\>/C$ is an $M$-f.s.\ sequence over $C$ if and only if $\<B_1,B_2\>/C$ is an $M$-f.s.\ sequence over $C$.
	\end{lemma}
	
	\begin{proof} Left to right is obvious.  For the converse, we need to show that $\tp(B_2/B_1AC)$ is finitely satisfied in $M$.
		To begin, by Proposition~\ref{baseextension}, choose  $B_2'\equiv_{B_1C} B_2$ with $\tp(B_2'/B_1AC)$ finitely satisfied in $M$.
		Note that $$B_1B_2'\equiv_C B_1B_2$$
		Also, since  $\<A,B\>/C$ is an $M$-f.s.\ sequence over $C$, we have both $\<A,B_1B_2\>/C$ and (by Shrinking) $\<A,B_1\>/C$ are $M$-f.s. sequences over $C$.
		By transitivity, the last statement, coupled with $\tp(B_2'/B_1AC)$ finitely satisfied in $M$, implies $\tp(B_1B_2'/AC)$ is finitely satisfied in $M$.
		Thus, by Lemma~\ref{stationary}, $$B_1B_2'\equiv_{AC} B_1B_2$$  As $\tp(B_2'/B_1AC)$ finitely satisfied in $M$, so is $\tp(B_2/B_1AC)$.
	\end{proof}

	\begin{lemma}  \label{startcongruence}
		Suppose $C\supseteq M$ is full and $\<A,B\>/C$ is an $M$-f.s.\ sequence over $C$.  Choose any $\abar_1,\abar_2$ from $A$ and $\bbar_1,\bbar_2$ from $B$ with
		$\tp(\abar_1/C)=\tp(\abar_2/C)$ and $\tp(\bbar_1/C)=\tp(\bbar_2/C)$.
		Then $\tp(\abar_1\bbar_1/C)=\tp(\abar_2\bbar_2/C)$.
	\end{lemma}
	
	\begin{proof}  Let $p=\tp(\abar_1/C)$.  As $\tp(\abar_1/C)=\tp(\abar_2/C)$, the map $f:C\abar_1\rightarrow C\abar_2$ fixing $C$ pointwise with $f(\abar_1)=\abar_2$
		is elementary.  To prove the Lemma, it suffices to show that $\bbar_2$ realizes $f(p)$.  
		
		To see this, let $\bbar^*$ be any realization of $f(p)$ (anywhere in $\C$).  Then $\tp(\abar_1\bbar_1/C)=\tp(\abar_2\bbar^*/C)$.
		From this it follows that  
		$\tp(\bbar^*/C)=\tp(\bbar_1/C)=\tp(\bbar_2/C)$, with the second equality by hypothesis.  
		But also:
		\begin{enumerate}
			\item  $\tp(\bbar^*/C\abar_2)$ is finitely satisfied in $M$ since $\tp(\abar_1\bbar_1/C)=\tp(\abar_2\bbar^*/C)$ and $\tp(B/AC)$ is finitely satisfied in $M$; and
			\item  $\tp(\bbar_2/C\abar_2)$ is finitely satisfied in $M$ since $\tp(B/AC)$ is finitely satisfied in $M$.
		\end{enumerate}
		Applying Lemma~\ref{stationary} to the last three statements implies $\tp(\bbar_2/\abar_2C)=\tp(\bbar^*/\abar_2C)$, i.e., $\bbar_2$ realizes $f(p)$.
	\end{proof}
	
	We glean two results from Lemma~\ref{startcongruence}.  The first bounds the number of types realized in an $M$-f.s.\ sequence, independent of either $|I|$ or $|A_i|$.
	
	\begin{lemma}  \label{fewtypes}  For any model $M$, for any $M$-f.s.\ sequence $\<A_i:i\in I\>$, and for every $i^*\in I$, $k\in\omega$,
		the number of complete $k$-types over $A_{<i^*}M$ realized in $A_{\ge i^*}$ is at most $\beth_2(|M|)$.
	\end{lemma}
	
	\begin{proof}  Because of Condensation, it suffices to prove that for any model $M$ and for any $M$-f.s.\ sequence $\<A,B\>$, at most $\beth_2(|M|)$ complete $k$-types are realized in $B$.
		To see this, choose a full $C_0\supseteq M$ with $|C_0|\le 2^{|M|}$.  By Proposition~\ref{baseextension},
		choose $C\supseteq M$ with $\tp(C/M)=\tp(C_0/M)$ and $\<A,B\>$ an $M$-f.s.\ sequence over $C$.
		Choose any $\bbar,\bbar'\in B^k$ with $\tp(\bbar/C)=\tp(\bbar'/C)$.  As both $\tp(\bbar/AC)$ and $\tp(\bbar'/AC)$ are finitely satisfied in $M$, it follows from Lemma~\ref{stationary}
		that $\tp(\bbar/AC)=\tp(\bbar'/AC)$.  As there are at most  $\beth_2(|M|)$ complete $k$-types over $C$, this suffices.
	\end{proof}
	
	The second is a refinement of the type structure of an $M$-f.s.\ sequence over a full $C\supseteq M$.
	
	\begin{definition}  An $M$-f.s.\ sequence $\<A_i:i\in I\>/C$ is an {\em order-congruence over $C$} if, for all $i^*\in I$, for all $i^*\le i_1<i_2\dots<i_n$, $i^*\le j_1<j_2<\dots j_n$ from $I$,
		and for all $\abar_k\in A_{i_k}, \bbar_k\in A_{j_k}$ satisfying
		$\tp(\abar_k/C)=\tp(\bbar_k/C)$ for $k=1,\dots,n$, we have
		$$\tp(\abar_1,\dots,\abar_n/CA_{<i^*})=\tp(\bbar_1,\dots,\bbar_n/CA_{<i^*})$$
	\end{definition}
	
	The following is essentially part of the statement of \cite{Hanf}*{Part I Lemma 2.6}.
	
	\begin{proposition}  \label{seqcongruence}  For every model $M$, every $M$-f.s.\ sequence $\<A_i:i\in I\>$ over any full $C\supseteq M$ is an order-congruence over $C$.
	\end{proposition}
	
	\begin{proof}  Fix $i^*$, $i^*\le i_1<\dots,i_n$, $i^*\le j_1<\dots<j_n$, $\abar_1,\dots,\abar_n$, and $\bbar_1,\dots,\bbar_n$ as in the hypotheses.
		By shrinking, for each $1\le k<n$, both  $\tp(\abar_{k+1}/C\abar_1,\dots,\abar_{k})$ and $\tp(\bbar_{k+1}/C\bbar_1,\dots,\bbar_{k})$ are finitely satisfied in $M$.
		As $C\supseteq M$ is full, by iterating Lemma~\ref{startcongruence} $(n-1)$ times, we have $\tp(\abar_1,\dots,\abar_n/C)=\tp(\bbar_1,\dots,\bbar_n/C)$.
		Also, both $\tp(\abar_1,\dots,\abar_n/CA_{<i^*})$ and $\tp(\bbar_1,\dots,\bbar_n/CA_{<i^*})$ are finitely satisfied in $M$, so
		$\tp(\abar_1,\dots,\abar_n/CA_{<i^*})=\tp(\bbar_1,\dots,\bbar_n/CA_{<i^*})$ by Lemma~\ref{stationary}.
	\end{proof}

	\subsection{$M$-f.s.\ sequences and indiscernibles}
	
	In this subsection, we explore the relation between $M$-f.s. sequences and indiscernibles. An $M$-f.s. sequence need not be indiscernible (for example, the tuples can realize different types), but when it is, it gives a special case of a Morley sequence in the sense of \cite{Pierre}.
	
	We first show indiscernible sequences can always be viewed as $M$-f.s. sequences over some model $M$.
	
		\textcolor{red}{In Lemma \ref{weakexistence}, the ``Furthermore'' sentence is false.}
	
	\begin{lemma} [extending \cite{Hanf}*{Part I Lemma 4.1}]\label{weakexistence} Suppose $(I,\le)$ is infinite and $\I=\<\abar_i:i\in I\>$ is indiscernible over $\emptyset$.  (For simplicity, assume $\lg(\abar_i)$ is finite).
		Then there is a model $M$ such that $\<\abar_i:i\in I\>$ is both indiscernible over $M$ and is an $M$-f.s.\ sequence.
		
		\st{Furthermore, if there is some $B$ such that $\I$ is indiscernible over each $b \in B$, then $M$ may be chosen so that $\I$ additionally remains indiscernible over $Mb$}.
	\end{lemma}
	\begin{proof}
		Expand the language to have built-in Skolem functions while keeping $\I$ indiscernible, and end-extend $\I$ to an indiscernible sequence of order-type $I + \omega^*$. (For the `Furthermore' sentence, note this can still be done so the result is indiscernible over each $b \in B$.) Let $I^*$ be the new elements added, and let $M$ be reduct of the Skolem hull of $I^*$ to the original language. (If $I + I^*$ were indiscernible over $b$, then $I$ is indiscernible over $Mb$.)
	\end{proof}
	
	Armed with this Lemma, we characterize when an infinite $\I=\<\abar_i:i\in I\>$ is both an $M$-f.s\ sequence and is indiscernible over $M$.  
	(A paradigm of an indiscernible sequence over $M$ that is not an $M$-f.s.\ sequence is where $M$ is an equivalence relation with infinitely many, infinite classes and 
	$\<a_i:i\in\omega\>$ is a sequence from some $E$-class not represented in $M$.)

	\begin{lemma} \label{charMorley}  An infinite sequence $\I=\<\abar_i:i\in I\>$ of $n$-tuples is both indiscernible over $M$ and an $M$-f.s.\ sequence if and only if
		there is an ultrafilter $\U$ on $M^n$ such that $\tp(\abar_i/MA_{<i})=Av(\U,MA_{<i})$ for every $i\in I$.
	\end{lemma}
	
	\begin{proof}  Right to left is clear, so assume $\I$ is both indiscernible over $M$ and an $M$-f.s.\ sequence.  As $(I,\le)$ is infinite, it contains either an ascending or descending $\omega$-chain.  For definiteness, choose $J\subseteq I$ of order type $\omega$.  To ease notation, we write $\abar_k$ in place of $\abar_{j_k}$.  For each $k\in \omega$ and each formula $\phi(\xbar,\bbar)\in\tp(\abar_k/MA_{<k})$, let $S^k_{\phi(\xbar,\bbar)}=\set{\mbar\in M^n:\C\vDash\phi(\abar_k,\bbar)}$.  As $\<\abar_k:k\in\omega\>$ is indiscernible over $M$, 
		$S^k_{\phi(\xbar,\bbar)}=S^{\ell}_{\phi(\xbar,\bbar)}$ for all $\ell\ge k$, and because it is an
		$M$-f.s.\ sequence,
		$\bigcup\set{S^k_{\phi(\xbar,\bbar)}:k\in\omega, \phi(\xbar,\bbar)\in\tp(\abar_k/MA_{<k})}$ has the finite intersection property.  Choose any ultrafilter $\U$ on $M^n$ containing every $S^k_{\phi(x,\bbar)}$.
		Thus, for any $k\le \ell<\omega$ and $\phi(\xbar,\bbar)$ with $\bbar\subseteq A_{<k}$, 
		$$\phi(\xbar,\bbar)\in\tp(\abar_\ell/MA_{<\ell})\quad\Leftrightarrow \quad S_{\phi(\xbar,\bbar)}\in\U \quad \Leftrightarrow \quad \phi(\xbar,\bbar)\in Av(\abar_\ell/MA_{<\ell})$$
		Finally, as $J\subseteq I$ and $\I$ is indiscernible over $M$, an easy induction on $\lg(\bbar)$ gives the result.
	\end{proof}
	
	Using Lemma~\ref{charMorley}, we obtain a strengthening of Lemma~\ref{weakexistence}.  The lemma below can be proved by modifying the proof of Lemma~\ref{baseextension}, but the argument here is fundamental enough to bear repeating.
	
	\begin{lemma}  \label{strongexistence}  If an infinite $\I=\<\abar_i:i\in I\>$ is both indiscernible over $M$ and an $M$-f.s.\ sequence, then for any $C\supseteq M$, there is $C'\supseteq M$ such that $\tp(C'/M)=\tp(C/M)$ and $\I$ is both indiscernible over $C'$ and an $M$-f.s.\ sequence over $C'$.  Thus, if $\I$ is an infinite, indiscernible sequence over $\emptyset$,
		then there is a model $M$ and a full $C\supseteq M$ such that $\I$ is both indiscernible over $C$ and an $M$-f.s.\  sequence  over $C$.
	\end{lemma}
	
	\begin{proof}  For the first sentence, given $\I$, $M$ and $C$, choose an ultrafilter $\U$ as in Lemma~\ref{charMorley}.  A routine compactness argument shows that we can
		find a sequence $\<\abar_i^*:i\in I\>$ such that $\tp(\abar^*_i/CA_{<i})=Av(\U,CA^*_{<i})$ for every $i\in I$.  As we also have $\tp(\abar_i:MA_{<i})=Av(\U,MA_{<i})$,
		an easy induction shows that $\tp(I/M)=\tp(I^*/M)$.  Now any $C'$ satisfying $\tp(IC'/M)=\tp(I^*C/M)$ suffices.
		
		For the second sentence, given $\I$, apply Lemma~\ref{weakexistence} to get an $M$ for which $\I$ is both indiscernible over $M$ and an $M$-f.s.\ sequence, and choose 
		any full $C\supseteq M$.  Then apply the first sentence to $\I$, $M$, and $C$.
	\end{proof}

	Next, we  recall the following characterization of indiscernibility.  
	The relevant concepts first appeared in the proof of Theorem~4.6 of \cite{Morley} and a full proof appears in \cite{Shc}*{Lemma~I,~2.5}.  
	
	\begin{lemma}  \label{old}  Suppose $\<A_i:i\in I\>$ is any sequence of sets indexed by a linear order $(I,\le)$ and let $B$ be any set. 
		For each $i\in I$,  fix a (possibly infinite)  enumeration $\abar_i$ of $A_i$ and let $p_i(\xbar)=\tp(\abar_i/BA_{<i})$.  Then $\<\abar_i:i\in I\>$ is indiscernible over $B$
		if and only if
		\begin{enumerate}
			\item  For all $i\le j$, $\abar_j$ realizes $p_i$; and 
			\item  Each $p_i$ does not split over $B$.
		\end{enumerate}
	\end{lemma}

	By contrast, if $C\supseteq M$ and $\<A_i:i\in I\>/C$ is an $M$-f.s.\ sequence over $C$, then (2) is satisfied, but (1) may fail.  
	In the case where $C\supseteq M$ is full, (1) reduces to a question about types over $C$.
	
	\begin{lemma} \label{charind}  Suppose $C\supseteq M$ is full and $\<A_i:i\in I\>$ is an $M$-f.s.\ sequence over $C$.
		Then $\<A_i:i\in I\>$ is indiscernible over $C$ if and only if $\tp(A_i/C)=\tp(A_j/C)$ for all $i,j\in I$.
	\end{lemma}
	
	\begin{proof}  Left to right is clear.  For the converse, fix $i<j$.  By Lemma~\ref{old}, it suffices to show $A_j$ realizes $p_i$.  But this is clear, as both
		$\tp(A_i/A_{<i}C)$ and $\tp(A_j/A_{<i}C)$ are finitely satisfied in $M$ and $\tp(A_i/C)=\tp(A_j/C)$.  Since $C\supseteq M$ is full, $\tp(A_j/A_{<i}C)=\tp(A_i/A_{<i}C)=p_i$
		by Lemma~\ref{stationary}.
	\end{proof}
	
	In terms of existence of such sequences, we have the following. 
	
	\begin{lemma}  \label{Morleyexist}  Suppose $C\supseteq M$ is full and $p(\xbar)\in S(C)$ is finitely satisfied in $M$.  Then for every $(I,\le)$ there is
		an $M$-f.s.\ sequence $\<\abar_i:i\in I\>$ over $C$ of realizations of $p$, hence is also indiscernible over $C$.
	\end{lemma}
	
	\begin{proof}  By compactness it suffices to prove this for $(I,\le)=(\omega,\le)$.
		By Fact~\ref{basic}(1), choose an ultrafilter $\U$ on $M^{\lg(\xbar)}$ and recursively let
		$\abar_i$ be a realization of $Av(\U,C\abar_{<i})$.  It is easily checked that $\<\abar_i:i\in\omega\>/C$ is an $M$-f.s.\ sequence over $C$ with $\tp(\abar_i/C)=p$ for each $i$.
		As $C\supseteq M$ is full,  it is also indiscernible over $C$ by Lemma~\ref{charind}.
	\end{proof}
	
	\section{The  f.s.\ dichotomy}
	
	We begin this section with the central dividing line of this paper.  Although unnamed, the concept appears in Lemma II 2.3 of \cite{Hanf}.
	
	\begin{definition}[f.s. dichotomy] \label{def:fs dich}  $T$ has the {\em f.s.\ dichotomy} if, for all models $M$, all finite tuples $\abar,\bbar \in \C$, if $\tp(\bbar/M\abar)$ is finitely satisfied in $M$, then for any singleton $c$,
		either $\tp(\bbar/M\abar c)$ or $\tp(\bbar c/M\abar)$ is finitely satisfied in $M$.
	\end{definition}
	
	It would be equivalent to replace $\abar$, $\bbar$ by sets $A,B \subset \C$ in the definition above, and this form will often be used.
	Much of the utility of the f.s.\ dichotomy is via the following extension lemma.

	\begin{lemma} [\cite{Hanf}*{Part I Claim 2.4}] \label{onepoint}  Suppose $T$ has the f.s.\ dichotomy and $\<A_i:i\in I\>$ is any $M$-f.s.\ sequence.  Then for every $c\in \C$, there is a simple extension
		$\<A_j':j\in J\>$ of
		$\<A_i:i\in I\>$ that includes $c$ that is also an $M$-f.s.\ sequence. Moreover, if $(I,\le)$ is a well-ordering with a maximum element, we may take $J=I$.
	\end{lemma}
	
	\begin{proof}  Fix any  $M$-f.s.\ sequence $\<A_i:i\in I\>$ and choose any  singleton $c\in \C$.  Let $I_0\subseteq I$ be the maximal initial segment of $I$ such that
		$\tp(c/A_{I_0}M)$ is finitely satisfied in $M$.  Note that $I_0$ could be empty or all of $I$.  If the minimum element of $(I\setminus I_0)$ exists, name it $i^*$ and take $J = I$; 
		otherwise, let $J=I\cup\{i^*\}$, where $i^*$ is a new element realizing the cut $(I_0, I \bs I_0)$ and put $A_{i^*} = \emptyset$.
		
		Let $A_{i^*}':=A_{i^*}\cup\set{c}$ and $A_{i}':=A_i$ for all $i\neq i^*$.  We claim that $\<A_i':i\in J\>$ is an $M$-f.s.\ sequence.   
		To see this, note that $A_{<i^*}'=A_{<i^*}=A_{I_0}$ while $A'_{<j}$ properly extends $A_{I_0}$ for any $j>i^*$.  Thus, $\tp(c/A_{<i^*}M)$ is finitely satisfied in $M$, but
		$\tp(c/A_{<j}M)$ is not for every $j>i^*$.
		
		We first show that
		$\tp(A_{i^*}c/A_{<i^*}M)$ is finitely satisfied in $M$. This is immediate if $A_{i^*} = \emptyset$. If not and $A_{i^*} \neq \emptyset$, by the f.s.\ dichotomy (with $A_i^*$ as $\bbar$, $A_{<i^*}$ as $\abar$, and $c$ as $c$), we  have that $\tp(A_{i^*}/cA_{<i^*}M)$ is finitely satisfied in $M$.
		But this, coupled with $\tp(c/A_{<i^*}M)$ is finitely satisfied in $M$, would imply $\tp(A_{i^*}c/A_{<i^*}M)$ is finitely satisfied in $M$, a contradiction.
		
		To finish, we show that for every $j> i^*$, $\tp(A_{j}/cA_{<j}M)$ is finitely satisfied in $M$.  Again, if this were not the case, the f.s.\ dichotomy would 
		imply $\tp(cA_j/A_{<j}M)$ is finitely satisfied in $M$.  But then, by Shrinking, we would have $\tp(c/A_{<j}M)$ finitely satisfied in $M$, contradicting our choice above.
		
		For the `Moreover' sentence, the only concern is if $\tp(c/A_IM)$ is finitely satisfied in $M$. But in this case we may take $i^*$ to be the maximal element of $I$, rather than a new element in $J \bs I$.
	\end{proof}
	
	It is evident that Lemma~\ref{onepoint} extends to $M$-f.s.\ sequences $\<A_i:i\in I\>/C$ over an arbitrary base $C\supseteq M$.
	
	In \cite{BS}*{Theorem 4.2.6}, the f.s. dichotomy appears as a statement about the behavior of forking rather than non-forking. Namely, forking dependence is totally trivial and transitive on singletons. We may derive similar consequences for dependence from the f.s. dichotomy. This is stated in \cite{Blum}*{Corollary 5.22}, although missing the necessary condition of full $C$.
	
	\begin{lemma}
		Suppose $T$ has the f.s. dichotomy, and fix  $\abar,\bbar,c,M\preceq\C$.
		\begin{enumerate}
			\item If $\tp(\abar/Mc)$ and $\tp(c/M\bbar)$ are not finitely satisfiable in $M$, then neither is $\tp(\abar/M \bbar)$.
			\item $\tp(\abar/M\bbar)$ is  finitely satisfiable in $M$ if and only if  $\tp(a/M\bbar)$ is finitely satisfied in $M$ for all $a \in \abar$.
			\item Let $C \supseteq M$ be full. Then $\tp(\abar/C\bbar)$ is  finitely satisfiable in $M$ if and only if  $\tp(a/Cb)$ is finitely satisfied in $M$  for all $a\in\abar$,  $b \in \bbar$.
		\end{enumerate}
	\end{lemma}
	\begin{proof}
		$(1)$ If $\tp(\abar/M \bbar)$ is finitely satisfiable in $M$, then by the f.s. dichotomy either $\tp(\abar/M \bbar c)$ or $\tp(\abar c/M \bbar)$ is as well. Shrinking then gives a contradiction.
		
		$(2)$ By induction on $\lg(\abar)$.  Left to right is immediate, so assume $\tp(a/M\bbar)$ is finitely satisfied in $M$ for every $a\in\abar$.  Write $\abar=\abar' a^*$.  By induction we
		may assume $\tp(\abar'/M\bbar)$ is finitely satisfied in $M$.  By the f.s.\ dichotomy, either $\tp(\abar' a^*/M\bbar)$ is finitely satisfied in $M$ and we are done immediately, or else
		$\tp(\abar'/M\bbar a^*)$ is finitely satisfied in $M$, and we finish using transitivity from Fact~\ref{basic}.

		$(3)$ Left to right is immediate by Shrinking, so assume $\tp(a/Cb)$ is finitely satisfied in $M$ for every $a\in\abar$ and $b\in\bbar$.  It follows from (2) that $\tp(\abar/Cb)$ is finitely satisfied in $M$ for every $b\in \bbar$.  To conclude that $\tp(\abar/C\bbar)$ is finitely satisfied in $M$, we argue by induction on $\lg(\bbar)$.
		Let $\bbar = \bbar' b^*$, and by induction assume the statement is true for $\bbar'$. 
		
		By the f.s. dichotomy, either $\tp(\abar/M\bbar' b^*)$ or $\tp(\abar'b^*/M\bbar')$ is finitely satisfiable in $M$. In the first case we are finished immediately, and in the second 
		we finish by invoking  Lemma \ref{1.6}.
	\end{proof}
	
	In the stable case, forking dependence is symmetric as well and so yields an equivalence relation on singletons, which is used in \cite{BS} to decompose models into trees of submodels. In general, the f.s. dichotomy shows finite satisfiability yields a quasi-order on singletons when working over a full $C \supseteq M$. Taking the classes of this quasi-order in order naturally gives an irreducible decomposition of $\C$ over $M$ in the sense of the next subsection, but we sometimes wish to avoid having to work over a full $C \supseteq M$.
	
	\subsection{Decompositions of models}
	
	In this subsection, we characterize the f.s.\ dichotomy in terms of extending partial decompositions to full decompositions of models.
	
	\begin{definition}[$M$-f.s. decomposition] \label{def:decomp}  Suppose $X \subseteq \C$ is any set. 
		\begin{itemize}
			\item   A {\em partial $M$-f.s.\ decomposition of $X$} is an $M$-f.s.\ sequence $\<A_i:i\in I\>$ with $\bigcup_{i\in I} A_i\subseteq X$.
			\item  An {\em $M$-f.s.\ decomposition of $X$} is a partial $M$-f.s.\ decomposition with $\bigcup_{i\in I}A_i=X$.
			\item  An $M$-f.s.\ decomposition of $X$ is {\em irreducible} if, for every $i\in I$ and for every $a,b\in A_i$, neither $\tp(a/MA_{<i}b)$ nor $\tp(b/MA_{<i}a)$ are finitely satisfied in $M$.
		\end{itemize}
	\end{definition}

	By iterating  Lemma~\ref{onepoint} for every $c\in X$ for a given set $X$ we obtain:
	
	\begin{lemma}  \label{secondpoint}  Suppose that $T$ has the f.s.\ dichotomy.  For any set $X\subseteq \C$ and any $C\supseteq M$, any partial   $M$-f.s.\ decomposition 
		$\<A_i:i\in I\>/C$ of $X$ has a 
		simple extension $\<A_j':j\in J\>/C$ to an $M$-f.s.\ decomposition of $X$ over $C$.  If the sets $\set{A_i:i\in I}$ were pairwise disjoint, we may choose $\set{A_j':j\in J}$ to be pairwise disjoint as well.  
		Furthermore, if $(I,\le)$ is a well-ordering with a maximum element, we may take $J=I$.
	\end{lemma}

		\begin{proposition}  \label{decomp}  $T$ has the f.s.\ dichotomy if and only if for all models $M,N\preceq \C$ (we do not require $M\subseteq N$)	every partial $M$-f.s.\ decomposition  $\<A_i:i\in I\>$ of $N$ has an irreducible  $M$-f.s.\ decomposition of $N$ extending it.
	\end{proposition}
	
	\begin{proof} 
		($\Leftarrow$) Suppose partial decompositions extend, and let $\abar, \bbar, c \in \C, M \prec \C$ with $\tp(\bbar/\abar)$ finitely satisfiable in $M$, and let $N \prec \C$ with $\abar, \bbar, c \in N$. So $\< \abar, \bbar \>$ is an $M$-f.s. sequence, and can be extended to an $M$-f.s. decomposition of $N$. After Shrinking, we obtain the conclusion of the f.s.-dichotomy.
		
		($\Ra$) Given a partial  $M$-f.s.\ decomposition  $\<A_i:i\in I\>$ of $N$, apply Lemma~\ref{secondpoint} to get a simple extension
		$\<B_k:k\in K\>$ that is an $M$-f.s.\ decomposition of $N$.  By Zorn's Lemma, it will suffice to show that if there is $k \in K$ with $a,b \in B_k$ and $\tp(b/B_{<k}a)$ is finitely satisfiable in $M$, then we may blow up the sequence so that $a$ and $b$ are in distinct parts. Given such $a, b \in B_k$, we have $\<a,b\>/B_{<k}$ is a partial $M$-f.s. decomposition of $B_k$ over $B_{<k}$. Extending this to a full decomposition of $B_k$ and then applying Lemma \ref{blow} to prepend $B_{<k}$ and append $B_{>k}$ gives the result.
	\end{proof}

	\begin{remark} {\em  We could do the above proof over some full $C \supseteq M$ to obtain an irreducible $M$-f.s. decomposition that is also an order-congruence. }
	\end{remark}
	
	We note that at the end of \cite{BS}, Baldwin and Shelah conjecture that models of monadically NIP theories should admit tree decompositions like those they describe for monadically stable theories, but with order-congruences in place of full congruences.
	
	\subsection{Preserving indiscernibility}
	
	We begin with some definitions.  The definition of dp-minimality given here may be non-standard, but it is proven equivalent to the usual definition with Fact~2.10 of
	\cite{DGL}. 
	
	\begin{definition}[Indiscernible-triviality and dp-minimality]  \label{def:ind triv}
		The first definition is meant to recall trivial forking.
		\begin{itemize}  
			\item  $T$ has  {\em indiscernible-triviality} if for every infinite indiscernible sequence $\I$ and every set $B$ of parameters, if $\I$ is indiscernible over each $b \in B$ then $\I$ is indiscernible over $B$.
			\textcolor{red}{\item  $T$ has  {\em endless indiscernible triviality} if for every infinite indiscernible sequence $\I$ without endpoints and every set $B$ of parameters, if $\I$ is indiscernible over each $b \in B$ then $\I$ is indiscernible over $B$.}
			\item  $T$ is {\em dp-minimal} if, for all indiscernible  sequences $\I=\<\abar_i:i\in I\>$ over any set $C$, every $b\in\C$ induces a finite partition
			of the index set into convex pieces 
			$\I=I_1\ll I_2\ll \dots\ll I_n$, with at most two  $I_j$ infinite
			and every $\I_j=\<\abar_i:i\in I_j\>$ is indiscernible over $Cb$.
		\end{itemize}
	\end{definition}

	As mentioned in the introduction, the notion of a theory admitting coding was the central dividing line of \cite{BS}. We weaken the definition here to allow the sequences to consist of tuples. Note that even the theory of equality would admit the further weakening of also allowing $C$ to consist of tuples.
	
	\begin{definition} [Admits coding (on tuples)] \label{def:tuple code}
		A theory $T$ {\em admits coding on tuples} if there is a formula  $\phi(\xbar, \ybar, z)$ (with parameters $\dbar$), sequences $\II =\<\abar_i : i \in I\>, \JJ = \<\bbar_j: j \in J\>$, 
		indexed by countable, dense orderings $(I,\le),(J,\le)$, respectively, and
		a set $\set{c_{i,j}| i \in I, j \in J}$, such that  $\II$ is indiscernible over $\bigcup \JJ \cup \dbar$,  $\JJ$ is indiscernible over $\bigcup \II \cup \dbar$, and  $\C \models \phi(\abar_i, \bbar_j, c_{k, l}) \iff (i,j) = (k,l)$.
		
		We call $\<\abar_i : i \in I\>, \<\bbar_j: j \in J\>, \set{c_{i,j}| i \in I, j \in J}, \phi(\xbar, \ybar, z)$ a {\em tuple-coding configuration}, and let $A = \bigcup\set{\abar_i : i \in I}, B = \bigcup \set{\bbar_j: j\in J}$ and $C = \set{c_{i,j}| i \in I, j \in J}$.
		
		$T$ {\em admits coding} if we may take $\II$ and $\JJ$ to be sequences of singletons.
	\end{definition}
	
	A convenient variant for this subsection is a {\em \joined}, which consists of a 
	formula (with parameters) $\phi(\xbar, \ybar, z)$, a sequence $\<\abar_i : i \in I\>$ indiscernible over the parameters of $\phi$, 
	indexed by an infinite linear order $(I,\le)$, and a set $\set{c_{i,j} | i < j \in I}$ such that for $i < j$, $\C \models \phi(\abar_i, \abar_j, c_{k, l}) \iff (i,j) = (k,l)$.
	Given a \joined, indexed by a countable, dense $(I,\le)$, we may construct a tuple-coding configuration by keeping $\phi(\xbar,\ybar,z)$ fixed, choosing open intervals $I', J' \subseteq I$ with $I' \ll J'$, and
	letting $\I'=\<\abar_i:i\in I'\>$ and $\J'=\<\abar_j:j\in J'\>$.   Conversely, given a tuple-coding configuration with $I = J$, we may construct a \joined\  by considering the indiscernible sequence whose elements are $\<\abar_i\bbar_i : i \in I\>$, restricting $C$ to elements $c_{i,j}$ with $i < j$, and replacing $\phi$ by $\phi^*(\xbar\xbar', \ybar\ybar', z) := \phi(\xbar, \ybar', z)$. 
	
	The following configuration appears in \cite{Hanf}*{Part II Lemma 2.2}, and will appear as an intermediate between a failure of the f.s. dichotomy and a tuple-coding configuration.
	
	\begin{definition}  A {\em pre-coding configuration} consists of a $\phi(\xbar,\ybar,z)$ with parameters and a sequence $\I=\<\dbar_i:i\in \Q\>$, indiscernible over the parameters of $\phi$, such that for some (equivalently, for every) $s<t$ from $\Q$, there is $c\in\C$ such that
		\begin{enumerate}
			\item  $\C\models\phi(\dbar_s,\dbar_t,c)$;
			\item  $\C\models \neg\phi(\dbar_s,\dbar_v,c)$ for all $v>t$; and
			\item $\C\models \neg\phi(\dbar_u,\dbar_t,c)$ for all $u<s$.
		\end{enumerate}
	\end{definition}
	
	We show the equivalence of the existence of these notions with the proposition below.   
	The proof of $(4) \Rightarrow (1)$ in the following is essentially from \cite{Hanf}*{Part II Lemma 2.3}, while $(3) \Ra (4)$ is based on \cite{ST}*{Lemma C.1}. The idea of $(4) \Ra (1)$ is that when working over a full $D \supseteq M$, types have a unique ``generic'' extension by Lemma \ref{stationary}. In a failure of the f.s. dichotomy, the extension of $\tp(c/D)$ to $\tp(c/Dab)$ is non-generic, and so $c$ can in some sense pick out $a$ and $b$ from a suitable sequence.
	
		\textcolor{red}{In Proposition \ref{preservingindisc}, ``indiscernible-triviality'' has been replaced with ``endless indiscernible triviality''. The adjustments needed for the proof of $(2) \Rightarrow (3)$ are minor, but $(1) \Rightarrow (2)$ needs an entirely new proof. These appear in Appendix \ref{app:indtriv}. We have also taken the opportunity to insert into $(3) \Rightarrow (4)$ a line about Ramsey and compactness that was omitted.}
	
	\begin{proposition}  \label{preservingindisc}  The following are equivalent for any theory $T$.
		\begin{enumerate}
			\item  $T$ has the f.s.\ dichotomy.
			\item  $T$ is dp-minimal and has \textcolor{red}{endless} indiscernible triviality.
			\item $T$ does not admit coding on tuples.
			\item  $T$ does not have a pre-coding configuration.
		\end{enumerate}
	\end{proposition}
	
	\begin{proof}  $(1)\Rightarrow(2)$:  Suppose $T$ has the f.s.\ dichotomy.  We begin with showing $T$ is dp-minimal.  Choose an indiscernible $\I=\<\abar_i:i\in I\>$ over a set $D$
		and any element $b\in \C$.  Applying Lemma~\ref{strongexistence} to $\I$ (in the theory $T_D$ naming constants for each $d\in D$) choose a model $M\supseteq D$ and a full set $C\supseteq M$ such that
		$\I$ is both indiscernible over $C$ and an $M$-f.s.\ sequence over $C$.  
		As in the proof of Lemma~\ref{onepoint}, choose a maximal initial segment $I_0\subseteq I$ such that
		$\tp(b/A_{I_0}C)$ is finitely satisfied in $M$.  If $(I\setminus I_0)$ has a minimal element $i^*$, let $I_1=(I\setminus (I_0\cup\set{i^*}))$, and let $I_1=I\setminus I_0$ otherwise.
		As $C$ is full, in either case we have an order-congruence of $\I\cup\set{b}$, so both $I_0$ and $I_1$ are indiscernible over $Cb$, which suffices.
		
		\st{Next, we show indiscernible-triviality.  We may assume $\I=(\abar_i:i\in\Q)$ is ordered by $(\Q,\le)$. The argument here is more involved, as given an infinite, indiscernible $\I$ and a set $B$ over which $\I$ is indiscernible over each $b\in B$,
		we cannot apply {Lemma~\ref{strongexistence}} and maintain the indiscernibility over each $b \in B$.  However, the proof of {Lemma~\ref{weakexistence}} allows
		us to find a model $M$ such that $\I$ is an $M$-f.s. sequence and is indiscernible over $Mb$ for every $b\in B$.  Now, call an element $b\in B$ {\em high} if
		$\tp(b/MI)$ is finitely satisfied in $M$.   By indiscernibility, if $\tp(b/MA_{<i})$ is finitely satisfied in $M$ for any $i$, then $b$ is high.  Let $B_1\subseteq B$
		denote the set of high elements, and let $B_0:=B\setminus B_1$ be the
		{\em low} (i.e., not high) elements of $B$.  As $T$ satisfies the f.s. dichotomy, \mbox{Lemma \ref{secondpoint}} implies $\I=(\abar_i:i\in\Q)$ has a simple extension (\mbox{Definition \ref{ext})} to an $M$-f.s. sequence with universe $(\bigcup \I)B$.
		Moreover, any such simple extension will condense to
		$B_0\smallfrown \I \smallfrown B_1$.}
		
		\st{Using this, we  argue that $\I$ is indiscernible over $MB$ in two steps.  
		First, we argue that $\I$ is indiscernible over $MB_0$.  To see this, fix $i<j$ from $\Q$ and let $p_i=\tp(\abar_i/A_{<i}MB_0)$.
		From the previous paragraph, $\I$ is an $M$-f.s. sequence over $MB_0$. So $p_i$ does not split over $M$, and so by \mbox{Lemma~\ref{old}} it suffices to prove that $\abar_j$ realizes $p_i$.
		Choose any $\phi(\xbar,\bbar,\mbar)\in p_i$ with $\mbar$ from $M$ and $\bbar$ from $B_0$. 
		To see that $\C\models\phi(\abar_j,\bbar,\mbar)$, choose an automorphism $\sigma\in Aut(\C)$ fixing $M$ and an initial segment 
		$(\abar_i:i\in I_0)$ pointwise that induces an order-preserving 
		permutation of $\I$ with $\sigma(\abar_i)=\abar_j$.  Clearly, $\C\models\phi(\abar_j,\sigma(\bbar),\mbar)$.
		It is easily seen that for every singleton $b'\in \sigma(\bbar)$,  $\I$ is indiscernible over $Mb'$ and, as $\sigma$ fixes
		$A_{I_0}$ pointwise, $b'$ is also low.  Thus, any simple extension to $(\bigcup \I)MB_0\sigma(\bbar)$ will condense to 
		$B_0\sigma(\bbar)\>\smallfrown(\abar_i:i\in\Q)\smallfrown B_1$.    In particular $\tp(\abar_j/M\bbar\sigma(\bbar))$ is finitely satisfied in $M$.  
		Thus, if
		$\C\models\neg\phi(\abar_j,\bbar,\mbar)$, by finite satisfiability there would be $\nbar$ from $M$ such that
		$\C\models\neg\phi(\nbar,\bbar,\mbar)\wedge \phi(\nbar,\sigma(\bbar),\mbar)$, which is impossible since $\sigma$ fixes $M$ pointwise.
		Thus, $\I$ is indiscernible over $MB_0$.}
		
		\st{Finally, to see that $\I$ is indiscernible over $MB_0B_1$, choose any $i_1<\dots i_k$, $j_1<\dots<j_k$ from $\Q$,  $\bbar$ from $MB_0$, and $\cbar$ from $B_1$
		and assume by way of contradiction that $\C\models\psi(\abar_{i_1},\dots,\abar_{i_k},\bbar,\cbar)\wedge\neg\psi(\abar_{j_1},\dots,\abar_{j_k},\bbar,\cbar)$.
		Recall $B_0\smallfrown \II \smallfrown B_1$ is an $M$-f.s. sequence, so $\tp(\cbar/M(\bigcup \I)B_0)$ is finitely satisfied in $M$, and so the same formula is true with some $\mbar$ from $M$ replacing $\cbar$.
		But this contradicts that $\I$ is indiscernible over $MB_0$.  Thus, $\I$ is indiscernible over $MB$.}
		
		\medskip\noindent 
		$(2)\Rightarrow(3)$:  Assume (2) holds, but  (3) fails, so there is a \joined\  $\<\abar_i: i \in I\>, \set{c_{i,j} | i < j \in I}, \phi(\xbar, \ybar, z)$ with $(I, \le)$ countable, dense. By naming constants, we may assume $\phi$ has no parameters.
		Choose $i<j$ from $I$. By dp-minimality, $(I,\le)$ is partitioned into finitely many convex pieces, indiscernible over $c_{i,j}$,  with at most two  pieces infinite.
		
		If $\abar_i, \abar_j$ are in the same convex piece, then taking $i < k < j$ we get $\phi(\abar_k, \abar_j, c_{i,j})$, contradicting our configuration. So suppose $\abar_i$ and $\abar_j$ are in different pieces. Then one of the pieces must be infinite, so by symmetry suppose the piece $I'$ containing $\abar_i$ is. By indiscernible-triviality $(\abar_i : i \in I')$ is indiscernible over $\abar_j c_{i,j}$. But then picking some $k \in I' \bs \set{i}$ again gives $\phi(\abar_k, \abar_j, c_{i,j})$. \textcolor{red}{The convex piece $I'$ containing $\abar_i$ might not be endless. But in this case, the convex piece containing $\abar_j$ must be $I \bs I'$, which is endless. So we may conclude the argument substituting $\abar_j$ for $\abar_i$ and $I \bs I'$ for $I'$.}
		
		\medskip\noindent 
		$(3)\Rightarrow(4)$: Assume $(4)$ fails, as witnessed by $\phi(\xbar, \ybar, z)$, $\<\abar_i :i \in \Q\>$,  and $\set{c_{i,j} : i < j \in \Q}$. \textcolor{red}{By Ramsey and compactness, we may assume that the truth value of $\phi(\abar_i, \abar_j, c_{k, \ell})$ depends only on the order-type of $ijk\ell$.} Define a new sequence $(\bbar_i : i \in \Z)$ where $\bbar_i = \abar_{3i-1}\abar_{3i}\abar_{3i+1}$. Let 
		$$\psi(\xbar_-\xbar\xbar_+, \ybar_-\ybar\ybar_+, z) := \phi(\xbar, \ybar, z) \wedge \neg \phi(\xbar_-, \ybar, z) \wedge \neg \phi(\xbar, \ybar_+, z)$$
		Also let $d_{i,j} = c_{3i,3j}$.  It is easily verified that  $(\bbar_i : i \in \Z), \set{d_{i,j}| i < j \in \Z}, \psi$ is a \joined.
		That $T$ admits coding on tuples now follows by the remarks following Definition~\ref{def:tuple code}.
		
		\medskip\noindent 
		$(4)\Rightarrow(1)$: Assume that $\abar$, $\bbar$, $M$, and $c$ form a counterexample to the f.s.\ dichotomy, i.e., $\tp(\bbar/M\abar)$ is finitely satisfied in $M$, but neither
		$\tp(\bbar c/M\abar)$ nor $\tp(\bbar/M \abar c)$ are.  As $\<\abar,\bbar\>$ is an $M$-f.s.\ sequence, by Proposition \ref{baseextension} choose a full $D\supseteq M$ such that
		$\<\abar,\bbar\>/D$ is an $M$-f.s.\ sequence over $D$.  Note that by transitivity, we also have $\tp(\abar\bbar/D)$ finitely satisfied in $M$.  
		
		\begin{claim*}  There is $c'$ such that $\abar\bbar c'\equiv_M \abar\bbar c$ with $\tp(\abar\bbar c'/D)$ finitely satisfied in $M$.
		\end{claim*}
		
		\begin{claimproof}  We first argue that every finite conjunction of formulas from $\tp(\abar\bbar/D)\cup\tp(\abar \bbar c/D)$ is satisfied in $M$.
			To see this choose $\phi(\xbar,\ybar,\dbar)\in \tp(\abar\bbar/D)$ and
			$\psi(\xbar,\ybar,z)\in \tp(\abar\bbar c/M)$ ($\phi$ and $\psi$ may also have hidden parameters from $M$) and we will show that
			$\phi(\xbar,\ybar,\dbar)\wedge\psi(\xbar,\ybar,z)$ has a solution in $M$.  Let $\theta(\xbar,\ybar,\dbar):=\phi(\xbar,\ybar,\dbar)\wedge\exists z\psi(\xbar,\ybar,z)$.
			As $\tp(\abar\bbar/D)$ is finitely satisfied in $M$, $\theta(\mbar,\nbar,\dbar)$ holds for some $\mbar,\nbar$ from $M$.  Thus, as $M\preceq\C$ and
			$\C\models\exists z\psi(\mbar,\nbar,z)$, there is $k\in M$ such that $\psi(\mbar,\nbar,k)$ holds, which suffices.

			Thus, by Fact~\ref{basic}(2), there is a complete type $p(\xbar,\ybar,z)\in S(D)$ extending $\tp(\abar\bbar/D)\cup\tp(\abar \bbar c/D)$ that is finitely satisfied in $M$.
			As $\tp(\abar\bbar/D)\subseteq p$, there is an element $c'$ so that $(\abar,\bbar,c')$ realizes $p$, proving the claim.
		\end{claimproof}
		
		By Lemma~\ref{Morleyexist}, choose an $M$-f.s.\  sequence $\<\abar_i\bbar_i c_i:i\in \Q\>$ over $D$ of realizations of $p$
		that is indiscernible over $D$. 
		Fix $s<t$ from $\Q$.  Note that since $\abar,\bbar,c,M$ witness the failure of the f.s.\ dichotomy, neither $\tp(\bbar_s c_s/D\abar_s)$ nor $\tp(\bbar_s/D\abar_s c_s)$ are finitely satisfied in $M$.
		As notation, let $\abar_{<s}=\bigcup\set{\abar_i:i<s}$ and $\bbar_{>t}=\bigcup\set{\bbar_j:j>t}$.  
		Now, by Shrinking and Condensation,
		$$\<\abar_{<s},(\abar_s\bbar_s),\bbar_t,\bbar_{>t}\>\ \hbox{is an $M$-f.s.\ sequence of length 4 over $D$.}$$
		As $\tp(\bbar_s/D\abar_s)$ is finitely satisfied in $D$ and $D$ is full, by Lemma~\ref{1.6} $\<\abar_{<s},\abar_s,\bbar_s\>$ is an $M$-f.s.\ sequence over $D$, and so by Lemma~\ref{blow},
	 $$\<\abar_{<s},\abar_s, \bbar_s,\bbar_t,\bbar_{>t}\>\ \hbox{is an $M$-f.s.\ sequence of length 5 over $D$.}$$
	
		As this sequence is an  order-congruence, it follows that $\tp(\bbar_s/D\abar_{<s}\abar_s\bbar_{>t})=\tp(\bbar_t/D\abar_{<s}\abar_s\bbar_{>t})$,
		so we can choose $c^*$ such that 
		$$\bbar_s c_s\equiv_{D\abar_{<s}\abar_s\bbar_{>t}} \bbar_t c^*$$
		Thus, since  $\tp(\bbar_s/D\abar_s c_s)$ is not finitely satisfied in $M$, neither is $\tp(\bbar_t/D\abar_s c^*)$.  
		By contrast, for $j,j'>t$, as both $\tp(\bbar_j/D\abar_s c_s)$ and $\tp(\bbar_{j'}/D\abar_s c_s)$ are finitely satisfied in $M$ and are equal by
		Proposition~\ref{seqcongruence}, the same is true of $\tp(\bbar_j/D\abar_s c^*)=\tp(\bbar_{j'}/D\abar_s c^*)$.
		Dually, since $\tp(\bbar_s c_s/D\abar_s)$ is not finitely satisfied in $M$, neither is $\tp(\bbar_s c^*/D\abar_s)$; and 
		because $\tp(\bbar_s c_s/D\abar_i)$ is finitely satisfied in $M$ for every $i<s$, we also conclude $\tp(\bbar_s c^* \abar_i/D)=\tp(\bbar_s c^* \abar_{i'}/D)$ for all $i,i'<s$ by 
		Proposition~\ref{seqcongruence}.  
		
		Finally,  choose $\rho_1(\xbar,\ybar,z),\rho_2(\xbar,\ybar,z)\in\tp(\abar_s,\bbar_t, c^*/D)$ such that neither $\rho_1(\abar_s,\ybar,c^*)$ nor $\rho_2(\abar_s,\ybar,z)$ has realizations in $M$.  Then, letting $\dbar_i:=\abar_i\bbar_i$ for each $i\in\Q$, $\I=\<\dbar_i:i\in \Q\>$ is a pre-coding configuration with respect to $\rho_1\wedge\rho_2$ and $c^*$.
	\end{proof}
	
	\begin{remark} \label{r:tidy}  {\em 
			The following observation will be useful in Section \ref{s:fin}. A {\em tidy} pre-coding configuration $\II = \<\dbar_i : i \in \Q\>, \set{c_{s,t} : s < t \in \Q}, \phi(\xbar, \ybar, z)$ is one where $\C \models \neg \phi(\dbar_i, \dbar_j, e)$ for every $i < j$ and $e \in \bigcup \II$. The pre-coding configuration constructed in $(4) \Ra (1)$ is tidy, since, choosing $M$ so $\II$ is an $M$-f.s. sequence,  $\C \models \phi(\dbar_i, \dbar_j, e)$ implies $\tp(\dbar_j/\dbar_i e)$ and $\tp(\dbar_je/\dbar_i)$ are not finitely satisfiable in $M$. But if $e \in \dbar_k$ for $k \leq i$ then the former type is finitely satisfiable in $M$, and if $e \in \dbar_k$ for $k > i$, then the latter type is.
			
			The tidiness property extends to the \joined \ constructed in $(3) \Ra (4)$ and so ultimately to the tuple-coding configuration as well. That is, from a failure of the f.s. dichotomy, we construct a tuple-coding configuration $\II, \JJ, C, \phi$ with $\C \models \neg \phi(\abar_i, \bbar_j, e)$ for every $\abar_i \in \II, \bbar_j \in \JJ$, and $e \in \bigcup \II \cup \bigcup \JJ$.
		}
	\end{remark}

	\section{The main theorem}
	
	We recall the main theorem from the introduction. Note that whereas Clauses (1) and (2) discuss monadic expansions of $T$, Clauses (3)-(6) are all statements about $T$ itself.
	
		\textcolor{red}{In Theorem \ref{together}, ``indiscernible-triviality'' has been replaced with ``endless indiscernible triviality''.}
	
	\begin{theorem}  \label{together}  The following are equivalent for a complete theory $T$ with an infinite model.
		\begin{enumerate}
			\item  $T$ is monadically NIP.
			\item  No monadic expansion of $T$ admits coding.
			\item  $T$ does not admit coding on tuples.
			\item  $T$ has the f.s.\ dichotomy.
			\item  For all $M^* \models T$ and $M, N \preceq M^*$, every partial $M$-f.s.\ decomposition of $N$ extends to an (irreducible) $M$-f.s.\ decomposition of $N$.
			\item  $T$ is dp-minimal and has \textcolor{red}{endless} indiscernible triviality.
		\end{enumerate}
	\end{theorem}
	
	The equivalences of $(3)$-$(6)$ are by Proposition~\ref{decomp} and Proposition~\ref{preservingindisc}.
	We note that $(1)\Rightarrow(2)$ is easy:
	Choose a monadic expansion $\C^*$ that admits coding, say via an $L$-formula $\phi(x,y,z)$ defining a bijection from 
	the countable sets $A\times B\rightarrow C$. By adding a new unary predicate for a suitable $C_0 \subseteq C$, the formula $\psi(x,y):=\exists z \in C_0 \phi(x,y,z)$ can define the edge relation of an arbitrary bipartite graph on $A \times B$, and in particular of the generic bipartite graph. Thus, $T$ is not monadically NIP.

	Thus, it remains to prove that $(4)\Rightarrow(1)$ and $(2)\Rightarrow(3)$, which are proved in the next two subsections.
	
	\subsection{If $T$ has the f.s.\ dichotomy, then $T$ is monadically NIP}
	
	The type-counting argument in this section is somewhat similar to that in \cite{Blum}, showing that monadic NIP corresponds to the dichotomy of unbounded partition width versus partition width at most $\beth_2(\az)$. Both arguments use the tools from Sections 2 and 3 to decompose the model and count the types realized in a part of the decomposition over its complement. However, while Blumensath decomposes the model into a large binary tree, our decomposition takes a single step.

	\begin{definition}  Suppose $\I=\set{\abar_i:i\in I}$ is any sequence of pairwise disjoint tuples and suppose $N\supseteq \bigcup\I$ is any model.
		An {\em $\I$-partition of $N$} is any partition $N=\bigsqcup \set{A_i:i\in I}$ with $\abar_i\subseteq A_i$ for each $i\in I$.
	\end{definition}
	
	\begin{definition}
		For any $N\preceq\C$ and $A\subseteq N$, let {\em $\rtp(N,A)$} denote the number of complete types over $A$ {\em realized} by tuples in $(N\setminus A)^{<\omega}$.
	\end{definition}
	We will be primarily interested in the case where $A$ is very large, and $\rtp(N,A)$ is significantly smaller than $|A|$. The following lemma is similar to Lemma \ref{fewtypes}, removing the requirement that the partition is convex but adding a finiteness condition. 
	
	For the rest of this section, recall the notation $A_J = \bigcup_{j \in J} A_j$ from the first part of Definition \ref{AJ}.
	
	\begin{lemma}  \label{fshalf}  If $T$ has the f.s.\ dichotomy, then for every well-ordering $(I,\le)$ with a maximum element, for every indiscernible sequence $\I=\set{\abar_i:i\in I}$ of pairwise disjoint tuples and every $N\supseteq\bigcup\I$,
		there is an $\I$-partition $\set{A_i:i\in I}$ of $N$ such that $\rtp(N,A_J)\le \beth_2(|T|)$  for every finite $J\subseteq I$.
	\end{lemma}
	
	\begin{proof}  By Lemma \ref{strongexistence}, choose a model $M$ of size $|T|$ 
		and a full $C\supseteq M$ with $|C|\le 2^{|T|}$ for which $\<\abar_i:i\in I\>/C$ is an $M$-f.s.\ sequence over $C$. (Note that $N$ might not contain $M$.) By Lemma~\ref{secondpoint} choose a simple extension $\<A_i:i\in I\>$ of $\<\abar_i:i\in I\>$ with $\bigsqcup\set{A_i:i\in I} = N$.  Thus, 
		$\set{A_i:i\in I}$ is an $\I$-partition of $N$. 
		For a given finite $J \subseteq I$ and $\nbar \in N \bs A_J$, let $\nbar \subset A_{i_1} \cup \dots \cup A_{i_k}$ and let $\nbar_{i_j} = \nbar \cap A_{i_j}$.  As $\<A_i:i\in I\>/C$ is an order-congruence over $C$ by Lemma~\ref{seqcongruence}, $\tp(\nbar/CA_J)$ is determined by $\set{\tp(\nbar_{i_j}/C) : 1 \leq j \leq k}$ and the order type of the finite set $\set{i_1, \dots, i_k} \cup J$. There are only finitely many such order types, and as $|C|\le 2^{|T|}$, there are at most $\beth_2(|T|)$ complete types over $C$. So $\rtp(N,A_J)\le \beth_2(|T|)$ for every finite $J\subseteq I$.
	\end{proof}
	
	On the other hand, if a theory $T$ has the independence property, then no  uniform  bound can exist.
	
	\begin{lemma} \label{IP} Suppose that $T$ has IP, as witnessed by $\phi(\xbar,y)$ with $\lg(\xbar)=n$ and $\lg(y)=1$.
		For every $\lambda>2^{|T|}$ there is an order-indiscernible $\I=(a_i:i\le\lambda)$, and a  model $N\supseteq \I$ such that for every 
		$\I$-partition $\<A_i:i\le\lambda\>$ of $N$,  $\rtp(N,A_{I_0})\ge \lambda$ for some finite $I_0\subseteq (\lambda+1)$.
	\end{lemma}
	\begin{proof}
		In the monster model, choose an order-indiscernible $\I=(a_i:i\le\lambda)$ that is shattered, i.e., there is a set $Y=\set{\bbar_s:s\in\P(\lambda)}$ 
		such that $\phi(\bbar_s,a_i)$ holds if and only if $i\in s$.  
		Note that for distinct $\bbar,\bbar'\in Y$, there is some $a_i \in I$ such that $\tp_\phi(\bbar/a_i)\neq\tp_\phi(\bbar'/a_i)$.
		Let $N$ be any model containing $\I\cup \bigcup Y$ and
		let $\<A_i:i\le\lambda\>$ be any $\I$-partition of $N$.  As $|\I|=\lambda$, while $|Y|=2^\lambda$, by applying the pigeon-hole principle $n$ times
		(one for each coordinate of $\bbar$) one obtains $Y'\subseteq Y$, also of size $2^\lambda$, and a finite $I_0\subseteq I$ such
		that $\bbar\in (A_{I_0})^n$ for each $\bbar\in Y'$.  As $\lambda>2^{|T|}$ and there are at most $2^{|T|}$ types over $I_0$, we can find
		$Y^*\subseteq Y'$ of size $2^\lambda$ such that $\tp(\bbar/I_0)$ is constant among $\bbar\in Y^*$.  
		It follows that $\tp(\bbar/(I-I_0))\neq\tp(\bbar'/(I-I_0))$ for distinct $\bbar,\bbar'\in Y^*$.  
		As $Y^*\subseteq (A_{I_0})^n$, it follows that $\rtp(N,A_{I_0})\ge \lambda$.
	\end{proof}
	
	To show that the behaviors of Lemma~\ref{fshalf} and Lemma~\ref{IP} cannot co-exist, we get an upper bound on the number of types realized in a finite monadic expansion. Such a bound is easy for quantifier-free types, and the next lemma inductively steps it up to a bound on all types.
	The following two lemmas make no assumptions about $T$.

	For each $k\in\omega$, define an equivalence relation $\sim_k$ on $(N\setminus A)^{<\omega}$ by:  $\abar\sim_k\bbar$  if and only if $\lg(\abar)=\lg(\bbar)$ and
	$\tp_\phi(\abar/A)=\tp_\phi(\bbar/A)$ for every formula $\phi(\zbar)$ of quantifier depth at most $k$.  
	Clearly, $\tp(\abar/A)=\tp(\bbar)$ if and only if $\abar\sim_k \bbar$ for every $k$.    To get an upper bound on $\rtp(N,A)$, for each $k\in\omega$, let
	$r_k(N,A)=|(N\setminus A)^{<\omega}/\sim_k|$.  
	
	\begin{lemma} \label{rtp} For any $N\preceq \C$, $A\subseteq N$, and $k\in \omega$,
		$r_{k+1}(N,A)\le 2^{r_k(N,A)}$.  Thus, $\rtp(N,A)\le \beth_{\omega+1}(r_0(N,A))$.
	\end{lemma}
	
	\begin{proof}  The second sentence follows from the first as $\tp(\abar/A)=\tp(\bbar/A)$ if and only if $\abar\sim_k \bbar$ for every $k$.
		For the first sentence, we give  an alternate formulation of $\sim_k$ to make counting easier.  
		For each $k\in\omega$, let  $E_k$ be the equivalence relation on $(N\setminus A)^{<\omega}$ given by:
		\begin{itemize}
			\item  $E_0(\abar,\bbar)$ if and only if $\lg(\abar)=\lg(\bbar)$ and $\qftp(\abar/A)=\qftp(\bbar/A)$; and
			\item   $E_{k+1}(\abar,\bbar)$ if and only if $E_k(\abar,\bbar)$ and, for every $c\in (N\setminus A)$, there is $d\in (N\setminus A)$ such that $E_k(\abar c,\bbar d)$, and vice-versa,
		\end{itemize}
		
		For each $k$, let $c(k):=|(N\setminus A)^{<\omega}/E_k|$.  It is clear that $c(0)=r_0(N,A)$ and by the definition of $E_{k+1}$ we have $c(k+1)\le 2^{c(k)}$ for each $k$, so the lemma follows from the fact that 
		$E_k(\abar,\bbar)$ if and only if $\abar\sim_k \bbar$, whose verification amounts to proving the following claim.
		
		\begin{claim*}  If the quantifier depth of $\phi(\zbar)$ is at most $k$, then for all partitions $\zbar=\xbar\ybar$, for all $\ebar\in A^{\lg(\ybar)}$, and for all $\abar,\bbar\in (N\setminus A)^{\lg(\xbar)}$, if $E_k(\abar,\bbar)$, then $N\models\phi(\abar,\ebar)\leftrightarrow\phi(\bbar,\ebar)$.
		\end{claim*}
		
		\begin{claimproof}  By induction on $k$.  Say $\psi(\zbar):=\exists w \phi(w,\zbar)$ is chosen with the quantifier depth of $\phi$ is at most $k$.  Fix a partition $\zbar=\xbar\ybar$ and choose $\ebar\in A^{\lg(\ybar)}$, $\abar,\bbar\in (N\setminus A)^{\lg(\xbar)}$ with $E_{k+1}(\abar,\bbar)$.  Assume $N\models\exists w\phi(w,\abar,\ebar)$.  There are two cases.
			If there is some $h\in A$ such that $N\models\phi(h,\abar,\ebar)$, then $N\models\phi(h,\bbar,\ebar)$ by the inductive hypothesis.  On the other hand, if there is $c\in (N\setminus A)$
			such that $N\models\phi(c,\abar,\ebar)$, use $E_{k+1}(\abar,\bbar)$ to find $d\in (N\setminus A)$ such that $E_k(\abar c, \bbar d)$.  Thus, the inductive hypothesis implies
			$N\models\phi(d,\bbar,\ebar)$, so again $N\models\psi(\bbar,\ebar)$.
		\end{claimproof}
	\end{proof}
	
	The following transfer result is the point of the previous lemma.  Again, it will be used when $\rtp(N^+,A)$ is significantly smaller than $|A|$.
	
	\begin{lemma}  \label{transfercount}  Let $N\supseteq A$ be any model  and let $N^+=(N,U_1,\dots,U_k)$ be any expansion of $N$ by finitely many unary predicates.
		Then $\rtp(N^+,A)\le\beth_{\omega+1}(\rtp(N,A))$.
	\end{lemma}
	\begin{proof}
		For each $n$, expanding by $k$ unary predicates can increase the number of quantifier-free $n$-types by at most a finite factor, i.e. $2^k$, so $r_0(N^+, A)= r_0(N, A) \leq \rtp(N, A)$. The result now follows from Lemma \ref{rtp}.
	\end{proof}

	Finally, we combine the lemmas above to  obtain the goal of this subsection.
	
	\begin{proposition}  If $T$ has the f.s.  dichotomy, then $T$ is monadically NIP.
	\end{proposition}
	
	\begin{proof}  By way of contradiction assume that  $T$ is not monadically NIP, but has the f.s.  dichotomy. 
		Let $T^+$ be an expansion by finitely many unary predicates that has IP.
		Choose a  cardinal $\lambda > \beth_{\omega+1}(|T|)$.
		Let $N^+ \models T^+$ with $N^+ \supseteq \I=(a_i:i\le\lambda)$ as in Lemma \ref{IP},  so for any $\I$-partition of $N^+$ there is $I_0 \subseteq (\lambda+1)$ with $\rtp(N^+,A_{I_0}) > \beth_{\omega+1}(|T|)$.
		
		Let $N$ be the $L$-reduct of $N^+$.  As $\I$ remains $L$-order-indiscernible, and $T$ has the f.s.  dichotomy, choose an
		$\I$-partition
		$\<A_i:I\le\lambda\>$ of $N$ as in Lemma~\ref{fshalf}, so $\rtp(N,A_{J})\le \beth_2(|T|)$ for every $J \subseteq (\lambda+1)$. Since $N^+$ is a unary expansion of $N$,
		$\rtp(N^+,A_J)\le \beth_{\omega+1}(|T|)$ for every $J \subseteq (\lambda+1)$, by Lemma~\ref{transfercount}.
		This this contradicts our ability to find an $I_0 \subseteq (\lambda+1)$ from the previous paragraph for the chosen $\I$-partition of $N^+$.
	\end{proof}
	
	Lemma \ref{fewtypes} and the arguments in this subsection seem to indicate that, for a generalization of the structural graph-theoretic notion of neighborhood-width \cite{Gur} similar to Blumensath's generalization of clique-width \cite{Blum}, monadic NIP should correspond to a dichotomy between bounded and unbounded neighborhood-width.
	
	\subsection{From coding on tuples to coding on singletons}
	
	This subsection provides the final step,  $(2)\Rightarrow(3)$, in proving Theorem~\ref{together} by showing that if $T$ admits coding on tuples, then some monadic expansion admits coding (i.e., on singletons).  For the result of this subsection, since $T$ admitting coding on tuples immediately implies $T$ is not monadically NIP, we could finish by \cite{BS}*{Theorem 8.1.8}, which states that if $T$ has IP then this is witnessed on singletons in a unary expansion. But the number of unary predicates used would depend on the length of the tuples in the tuple-coding configuration, which would weaken the results of Section \ref{s:fin}.
	
	Deriving non-structure results in a universal theory from the existence of a bad configuration is made much more involved if the configuration can occur on tuples. If one is willing to add unary predicates, arguments such as that from \cite{BS} mentioned above will often bring the configuration down to singletons. A general result in this case is \cite{Blum}*{Theorem 4.6} that (under mild assumptions) there is a formula defining the tuples of an indiscernible sequence in the expansion adding a unary predicate for each ``coordinate strip'' of the sequence. The results of \cite{Sim21} indicate the configuration can often be brought down to singletons just by adding parameters, instead of unary predicates, but these arguments seem difficult to adapt to tuple-coding configurations. Another approach, which we use here, is to take an instance of the configuration where the tuples have minimal length, and argue that the tuples then in many ways behave like singletons. 
	
	\begin{definition}  \label{standard}
		Given a tuple-coding configuration $\II = \<\abar_i : i \in I\>, \JJ = \<\bbar_j: j \in J\>, \set{c_{i,j}| i \in I, j \in J}, \phi(\xbar, \ybar, z)$, indexed by disjoint countable dense orderings $(I,\le),(J\le)$,
		an order-preserving permutation of $(I,\le)$ (resp. $(J,\le)$) naturally gives rise to permutation of $A=\bigcup\II$ (resp. $B=\bigcup\JJ$; call such permutations of $A$ and $B$  {\em standard permutations}.
		
		A tuple-coding configuration as above is {\em regular} if  $\C \models \phi(\dbar, \ebar, c_{i,j}) \leftrightarrow \phi(\sigma(\dbar), \tau(\ebar), c_{i,j})$ whenever $\dbar \subseteq A, \ebar \subseteq B$ (including cases with $\dbar \not\in \II, \ebar \not\in \JJ$), $\sigma$ is a standard permutation of $A$ corresponding to an element of $Aut(I, \le)$ fixing $i$, and $\tau$ is a standard permutation of $B$ corresponding to an element of $Aut(J,\le)$ fixing $j$.
	\end{definition}
	
	By Ramsey and compactness, if $T$ admits coding on tuples via the formula $\phi(\xbar,\ybar,z)$, then it admits a regular tuple-coding configuration via the same formula $\phi(\xbar,\ybar,z)$.
	
	\begin{definition}
		Let $\<\abar_i : i \in I\>, \<\bbar_j: j \in J\>, \set{c_{i,j}|i \in I,j \in J}, \phi(\xbar, \ybar, z)$ be a tuple-coding configuration with $(I,\le), (J,\le)$ countable, dense. The pair $(\dbar, \ebar)$ with $\dbar \subseteq A, \ebar \subseteq B$ is a {\em witness} for $c_{k,\ell}$ if there are open intervals $I' \subseteq I, J' \subseteq J$ with $k \in I', \ell \in J'$ such that for all $k' \in I', \ell' \in J'$, we have $\C \models \phi(\dbar, \ebar, c_{k', \ell'}) \iff (k, \ell) = (k', \ell')$.
		
		A tuple-coding configuration {\em has unique witnesses up to permutation} if for every $c_{i,j} \in C$, the only witnesses for $c_{i,j}$ are of the form $(\sigma(\abar_i), \tau(\bbar_j))$ for some $\sigma$ a permutation of $\abar_i$ and some $\tau$ a permutation of $\bbar_j$
	\end{definition}
	
	\begin{lemma} \label{lemma:uw}
		Let $\<\abar_i : i \in I\>, \<\bbar_j: j \in J\>, \set{c_{i,j}| i,j \in I}, \phi(\xbar, \ybar, z)$ be a regular tuple-coding configuration for $T$, with $|\xbar|+|\ybar|$ minimal. Then this configuration has unique witnesses up to permutation.
	\end{lemma}
	\begin{proof}
		Suppose not, and let $(\dbar, \ebar)$ be a witness for $c_{i,j}$, such that $(\dbar, \ebar) \neq (\sigma(\abar_i), \tau(\bbar_j))$ for any $\sigma, \tau$. First, if either $\dbar \cap \abar_i = \emptyset$ or $\ebar \cap \bbar_j = \emptyset$, then regularity immediately implies that $(\dbar, \ebar)$ is not a witness. So let $\dbar^*$ be the subsequence of $\dbar$ intersecting $\abar_i$, and $\ebar^*$ the subsequence of $\ebar$ intersecting $\bbar_j$. Either $\dbar^* \neq \dbar$ or $\ebar^* \neq \ebar$ so assume the former.
		
		Let $I^* \subseteq I$ be an open interval such that $\dbar \cap \bigcup (\abar_i : i \in I^*) = \dbar^*$. Let $\phi^*(\xbar^*, \ybar, z)$ be the formula obtained by starting with $\phi(\dbar, \ybar, z)$, and then replacing the subtuple $\dbar^*$ with the variables $\xbar^*$; so we have plugged the elements of $\dbar\bs\dbar^*$ as parameters into $\phi$. For each $k \in I^*$, let $\abar^*_k$ be the restriction of $\abar_k$ to the coordinates corresponding to $\dbar^*$. Then $\<\abar^*_i : i \in I^*\>, \<\bbar_j: j \in J\>, \set{c_{i,j}|i \in I^*, j \in J}, \phi^*(\xbar^*, \ybar, z)$ is also a regular tuple-coding configuration, contradicting the minimality of $|\xbar|+|\ybar|$.
	\end{proof}
	
	The following Lemma completes the proof of Theorem~\ref{together}.

	\begin{lemma} \label{lemma:tuple to 1}
		Suppose $T$ admits coding on tuples. Then $T$ admits coding in an expansion by three unary predicates.
	\end{lemma}
	\begin{proof}
		Choose a tuple-coding configuration $$\bar{A}=\<\abar_i : i \in I\>, \bar{B} = \<\bbar_j: j \in J\>, C=\set{c_{i,j}| i \in I,j \in I}, \phi(\xbar, \ybar, z)$$ with $|\xbar|+|\ybar|$ as small as possible.
		By the remarks following Definition~\ref{standard}. we may assume this configuration is regular, so by 
		Lemma \ref{lemma:uw}, it has unique witnesses up to permutation.
		Let $L^*=L\cup\set{A,B,C}$ and let $\C^*$ be the expansion of $\C$ interpreting $A$ as $\bigcup \bar{A}$, $B$ as $\bigcup \bar{B}$, and $C$ as itself.
		Let 
		\begin{align*}
			&\phi^*(x, y, z) := A(x) \wedge B(y) \wedge C(z) \wedge \\
			&\exists \xbar' \subseteq A, \ybar' \subseteq B(\phi(x\xbar', y\ybar', z) \wedge \forall z' \in C(z' \neq z \ra \neg \phi(x\xbar', y\ybar', z')))
		\end{align*}
		
		Let $a_i$ be the first coordinate of $\abar_i$, and $b_j$ the first coordinate of $\bbar_j$. Then 
		$A_1=\set{a_i : i \in I}, B_1=\set{b_j: j \in J}$, and $C$ witness coding  in $T^*=Th(\C^*)$ via the $L^*$-formula $\phi^*(x, y, z)$.
	\end{proof}
	
	\section{Finite structures} \label{s:fin}

	In this section, we restrict the language $L$ of the theories we consider to be relational (i.e., no function symbols) with only finitely many constant symbols.
	
	\begin{definition} \label{def:fin}  For a complete theory $T$ and $M \models T$, $Age(M)$, the isomorphism types of finite substructures of $M$ does not depend on the choice of $M$, so we let
		$Age(T)$ denote this class of isomorphism types.  
		
		The {\em growth rate} of $Age(T)$ (sometimes called the profile or (unlabeled) speed) is the function $\varphi_T(n)$ counting the number of isomorphism types with $n$ elements in $Age(T)$.
		
		 We also investigate $Age(T)$
		under the quasi-order of embeddability.We say $Age(T)$ is {\em well-quasi-ordered (wqo)\/} if this class does not contain an infinite antichain, and we say
		$Age(T)$ is {\em $n$-wqo}  if $Age(M^*)$ is wqo for every expansion $M^*$ of any model $M$ of $T$ by $n$ unary predicates that partition the universe.
	\end{definition}
	
	The definition of $n$-wqo is sometimes given for an arbitrary hereditary class $\CC$ rather than an age, with $\CC$ $n$-wqo if the class $\CC^*$ containing every partition of every structure of $\CC$ by at most $n$ unary predicates remains wqo. Our definition is possibly weaker, but then its failure is stronger.
	
	\begin{example}  Let $T=Th(\Z, succ)$.  Then $Age(T)$ is wqo, but not 2-wqo, since $Age(T)$ contains arbitrarily long finite paths, and marking the endpoints of these paths with a unary predicate gives an infinite antichain.
		
		By contrast, if $T = Th(\Z, \leq)$, then $Age(T)$ can be shown to be $n$-wqo for all $n$.
	\end{example}

	The following lemma shows that when considering $n$-wqo, adding finitely many parameters is no worse than adding another unary predicate.
	
	\begin{lemma} \label{lemma:wqo const}
		Suppose $Age(M)$ is $(n+1)$-wqo. If $M^*$ is an expansion of $M$ by finitely many constants, then $Age(M^*)$ is $n$-wqo.
	\end{lemma}
	\begin{proof}
		Suppose an expansion by $k$ constants is not  $n$-wqo, as witnessed by an infinite antichain $\set{M^+_i}_{i \in \omega}$ in a language $L^+$ expanding the initial language by the $k$ constants and by $n$ unary predicates. Let $M^*_i$ be the structure obtained from $M^+_i$ by forgetting the $k$ constants, but naming their interpretations by a single new unary predicate. As $Age(M)$ is $(n+1)$-wqo, $\set{M^*_i}_{i \in \omega}$ contains an infinite chain $M^*_{i_1} \hookrightarrow  M^*_{i_2} \hookrightarrow\dots$
		under embeddings.   As there are only finitely many permutations of the constants, some embedding in the chain must preserve them, contradicting that $\set{M^+_i}_{i \in \omega}$ is an antichain. 
	\end{proof}
	
	In both Theorem \ref{thm:fin} and \ref{thm:coll}, the assumption that $T$ has quantifier elimination is only used to get that the formula witnessing that $T$ admits coding on tuples is quantifier-free, and the formula witnessing the order property in the stability part of Theorem \ref{thm:coll}, so the hypotheses of the theorems can be weakened to only these specific formulas being quantifier-free. This weakened assumption is used in \cite{ST}. From the proof of Proposition \ref{preservingindisc}, if the failure of the f.s. dichotomy is witnessed by quantifier-free formulas, then the formula witnessing coding on tuples will be quantifier-free as well.
	
		\textcolor{red}{In Theorem \ref{thm:fin}, the growth rate statement remains intact, but the claim that $Age(T)$ is not 4-wqo is unproven. See Appendix \ref{app:wqo}.}
	
	\begin{theorem} \label{thm:fin}
		If a complete  theory $T$ has quantifier elimination in a relational language with finitely many constants is not monadically NIP, then  $Age(T)$ has growth rate asymptotically greater than $(n/k)!$ for some $k \in \omega$ \st{and is not 4-wqo}.
	\end{theorem}
	\begin{proof}
		Since $T$ is not monadically NIP, let 
		$$\<\abar_i : i \in I\>, \<\bbar_j: j \in J\>, \set{c_{i,j}| i \in I, j \in J}, \phi(\xbar, \ybar, z)$$ 
		be a regular tuple-coding configuration with unique witnesses up to permutation. The only place we use $T$ has QE is to choose $\phi$ quantifier-free. Let $L^*$ expand by unary predicates for $A, B$, and $C$ as well as constants for the parameters of $\phi$, and let $\phi^*$ be as in the proof of \ref{lemma:tuple to 1}. Let $\AA \subseteq Age(T^*)$ be the set of finite substructures that can be constructed as follows.
		\begin{enumerate}
			\item Pick $X \subset_{fin} I$, $Y \subset_{fin} J$, and $E \subset X \times Y$.
			\item Start with $\set{\abar_i : i \in X} \cup \set{\bbar_j: j \in Y} \cup \set{c_{i, j} | (i,j) \in E}$.
			\item For every point $c_{i,j}$ added in the previous step, add the four elements $c_{i \pm \epsilon, j}$ and $c_{i, j \pm \epsilon}$, where $i \pm \epsilon$ are closer to $i$ than any other element of $X$, and $j \pm \epsilon$ are closer to $j$ than any other element of $Y$.
			\item Add the parameters of $\phi^*$.
		\end{enumerate}
		
		\begin{claim*}
			For any $M \in \AA$ and $a, b, c \in M$, $\C \models \phi^*(a, b, c) \iff M \models \phi^*(a,b,c)$.
		\end{claim*}
		\begin{claimproof}
			Since $\phi$ is quantifier-free, it remains to check that if the existential quantifiers in $\phi^*$ are witnessed in $\C$ and $\C \models \phi^*(a,b,c)$ then they are witnessed in $M$, and if the universal fails in $\C$ then it fails in $M$. From the unary predicates at the beginning of $\phi^*$, we may let $a \in \abar_i, b \in \bbar_j$, and $c=c_{k,\ell}$. If $\C \models \phi(a,b,c)$, the only tuple in $\C$ that can witness $\xbar'$ is the rest of the tuple $\abar_i$, which will be in $M$ because it only contains full tuples, and similarly for witnessing $\ybar'$. Since our configuration has unique witnesses up to permutation, if the universal quantifier fails in $\C$, this is witnessed by an element $c_{k', \ell'}$ with $i - \epsilon \leq k' \leq i+\epsilon$ and $j - \epsilon \leq \ell' \leq j+\epsilon$. By regularity, this failure is also witnessed by some element in $\set{c_{i \pm \epsilon, j}, c_{i, j \pm \epsilon}}$.
		\end{claimproof}
		
		Given a bipartite graph $G$ with $n$ edges and no isolated vertices, we may encode it as a structure $M_G \in \AA$ by starting with tuples $\abar_i$ for each point in one part and tuples $\bbar_j$ for each point in the other part, and including $c_{k, \ell}$ whenever we want to encode an edge between $\abar_k$ and $\bbar_{\ell}$. Note that $|M_G| = O(n)$, and this encoding preserves isomorphism in both directions. In the proof of \cite{Rap}*{Theorem 1.5}, the asymptotic growth rate of such graphs is shown to be at least $(n/5)!$, which gives the desired growth rate for $Age(T^*)$ with the constant $k$ depending on the length of the tuples in the tuple-coding configuration. Since expanding by finitely many unary predicates and constants increases the growth rate by at most an exponential factor, we also get the desired growth rate for $Age(T)$. 
		
		\st{Furthermore, if $M_H$ embeds into $M_G$, then $H$ must be a (possibly non-induced) subgraph of $G$. So we get that $Age(T^*)$ is not wqo by encoding even cycles. We expanded by three unary predicates, and by Lemma {\ref{lemma:wqo const}} the parameters may be replaced by another unary predicate while still preserving the failure of wqo, so we get that $Age(T)$ is not 4-wqo.}
	\end{proof}

	\begin{remark}  {\em
			There is a homogeneous structure, with automorphism group $S_\infty \Wr S_2$ in its product action, that is not monadically NIP and whose growth rate is the number of bipartite graphs with a prescribed bipartition, $n$ edges, and no isolated vertices. So the lower bound in this theorem cannot be raised above the growth rate of such graphs. Precise asymptotics for this growth rate are not known, although it is slower than $n!$ and  \cite{CPS}*{Theorem 7.1} improves Macpherson's lower bound to $(\frac{n}{\log{n}^{2 + \epsilon}})^n$ for every $\epsilon > 0$.
		}
	\end{remark}
	
	If Conjecture \ref{conj:hom} from the Introduction (in particular $(1) \Ra (2)$) is confirmed, then the lower bound on the growth rate in Theorem \ref{thm:fin} would also confirm \cite{Rap}*{Conjecture 3.5} that for homogeneous structures there is a gap from exponential growth rate to growth rate at least $(n/k)!$ for some $k \in \omega$.
	
	Theorem \ref{thm:fin} is somewhat surprising. Since passing to substructures can be simulated by adding unary predicates, it is clear that if $T$ is monadically tame, then $Age(T)$ should be tame. However, unary predicates can do more, so it seems plausible that $Age(T)$ could be tame even though $T$ is not monadically tame. Our next theorem gives some explanation for why this does not occur, at least when assuming quantifier elimination. 
	
	First we need to define stability and NIP for hereditary classes. The following definition is standard and appears, for example, in \cite{ST}*{\S 8.1}.

%
%
	
	\begin{definition}
	        For a formula $\phi(\xbar,\ybar)$ and a bipartite graph $G=(I,J,E)$, we say a structure $M$ {\em encodes $G$ via $\phi$} if
		 there are sets $A=\set{\abar_i | i \in I} \subseteq M^{|x|}, B=\set{\bbar_j | j \in J} \subseteq M^{|y|}$ such that 
		 $M\models\phi(\abar_i,\bbar_j)\Leftrightarrow G\models E(i,j)$.

		A class of structures $\CC$ has {\em IP} if there is some formula $\phi(\xbar, \ybar)$ such that for every finite, bipartite graph $G=(I,J,E)$, there is some $M_G\in\CC$
		encoding $G$ via $\phi$.  
		 Otherwise, $\CC$ is {\em NIP}.
		
		A class of structures $\CC$ is {\em unstable} if there is some formula $\phi(\xbar, \ybar)$ such that for every finite half-graph $G$, there is some
		$M_G\in\CC$ encoding $G$ via $\phi$.  Otherwise, $\CC$ is {\em stable}.
	\end{definition}
Equivalently, by compactness arguments, $\CC$ is NIP (resp. stable) if and only if every completion of $Th(\CC)$, the common theory of structures in $\CC$, is.
	Note that it suffices to witness that $\CC$ has IP or is unstable using a formula with parameters, since we can remove them by appending the parameters to each $\abar_i$.
	
	The sort of collapse between monadic NIP and NIP in hereditary classes observed in the next theorem occurs for binary ordered structures \cite{ST}, since there the formula giving coding on tuples is quantifier-free. It also occurs for monotone graph classes (i.e. specified by forbidding non-induced subgraphs), where NIP actually collapses to monadic stability, and agrees with nowhere-denseness \cite{AA}.

	\begin{theorem} \label{thm:coll}
		Suppose that a complete theory $T$  in a relational language with finitely many constants has quantifier elimination. Then $Age(T)$ is NIP if and only if $T$ is monadically NIP, and $Age(T)$ is stable if and only if $T$ is monadically stable.
	\end{theorem}
	\begin{proof}
		We first consider the NIP case.
		
		$(\Leftarrow)$ Suppose $Age(T)$ has IP, as witnessed by the formula $\phi(\xbar, \ybar)$. By compactness, there is a  model $N$ of the universal theory of $T$ in which $\phi$ encodes the generic bipartite graph. But then $N$ is a substructure of some $M \models T$, and naming a copy of $N$ in $M$ by a unary predicate $U$ and relativizing $\phi$ to $U$ gives a unary expansion of $M$ with IP.
		
		$(\Ra)$ Suppose $T$ is not monadically NIP, witnessed by a tuple-coding configuration $\II = (\abar_i : i \in I), \JJ = (\bbar_j: j \in J)$, $C = \set{c_{i,j}| i \in I, j \in J}, \phi(\xbar, \ybar, z)$, with $\phi$ quantifier-free and containing parameters $\mbar$. By Remark \ref{r:tidy}, we may also assume the configuration is tidy. For any bipartite graph $G$, let $M_G \in Age(T)$ contain $\mbar$, tuples from $\II$ and $\JJ$ corresponding to the two parts of $G$, and an element of $c_{i,j}$ for each edge of $G$ so that $R^*(\xbar, \ybar; \mbar) := \exists z \in C (\phi(\xbar, \ybar, z; \mbar))$ encodes $G$ on $\II(M_G) \times \JJ(M_G)$. But by tidiness, $R(\xbar, \ybar; \mbar) := \exists z (\phi(\xbar, \ybar, z; \mbar) \wedge z \not\in \mbar)$ encodes $G$  on $\II(M_G) \times \JJ(M_G)$ as well.
		
		For the stable case, the backwards direction is the same except using the infinite half-graph in place of the generic bipartite graph. For the forwards direction, if $T$ is unstable then by quantifier-elimination $Age(T)$ is also unstable. If $T$ is stable but not monadically stable, then by \cite{BS}*{Lemma 4.2.6} $T$ is not monadically NIP, so we are finished by the NIP case.
	\end{proof}

\appendix

\def\sigmai{\sigma^{(i)}}

\section{Corrigenda}

\subsection{Indiscernible triviality} \label{app:indtriv}  In Theorem 1.1 of \cite{MonNIP}, several equivalents of a theory being monadically NIP are given.  With the definition of indiscernible-triviality given there, (6) is not equivalent, as can be seen by Example~\ref{example} below.  However, by making a slight variation on the definition of indiscernible-triviality the equivalence of (6) with the other properties is maintained.   Call a linear order $(I,<)$ {\em endless} if it has neither a minimum nor a maximum element.  Clearly, any endless linear order is infinite.   

\begin{definition} \label{newIT}  A theory $T$ has  {\em endless indiscernible triviality} if for every {\bf endless} indiscernible sequence $\II=(\abar_i:i\in I)$ and every set $B$ of parameters, if $\II$ is indiscernible over each $b \in B$ then $\II$ is indiscernible over $B$.
\end{definition}

This is the same as the definition of indiscernible-triviality, except that {\em infinite} has been replaced by {\em endless}. 

With this note, we prove the following theorem.

\begin{theorem} \label{remain} {\em Replacing indiscernible-triviality by endless indiscernible triviality, the six statements described in \cite[Theorem~1.1]{MonNIP} are equivalent.}
\end{theorem}

Before launching into the proof of Theorem~\ref{remain}, we highlight what the problem was in \cite{MonNIP}.
The first issue is the Furthermore clause in \cite[Lemma 2.18]{MonNIP}, used in the proof of \cite[Proposition 3.11]{MonNIP}.   We thank James Hanson for providing a counterexample to this clause. The second issue is that in the proof of \cite[Proposition 3.11]{MonNIP}, we assumed that the failure of indiscernible triviality could be witnessed by a $\Q$-indexed sequence, obliterating the distinction between indiscernible triviality and endless indiscernible triviality. To see that (full) indiscernible triviality can fail in a monadically NIP theory, we thank Artem Chernikov for indicating the following example.  

\begin{example} \label{example}
Let $T_{dt}$ be the theory of dense meet-trees as in \cite[Section 2.3.1]{Pierre}. By \cite[Corollary 2.8]{parigot1982theories}, $T_{dt}$ is monadically NIP. (It is also fairly easy to check the quantifier-free type-counting criterion in \cite[Proposition 4.8]{MonNIP} over indiscernible sequences of singletons, which is sufficient.) Let $M \models T_{dt}$, let $\II = (a_i : i \in \omega)$ be a decreasing sequence, and let $b,b' \in M$ be such that $b,b' > a_0$ and $b \wedge b' = a_0$. Then $\II$ is indiscernible over $b$ and over $b'$, but not over $bb'$.
\end{example}

In the remainder of this section, we prove Theorem~\ref{remain}  and indicate     where the endlessness assumption  is used.  The following definition, which appears in \cite{Pierre}, is standard.

\begin{definition}  Two sequences $(\abar_i:i\in I)$ and $(\bbar_j:j\in J)$ (not necessarily of the same arities) are {\em mutually indiscernible over $C$} if
$(\abar_i:i\in I)$ is indiscernible over $C\cup\bigcup\{\bbar_j:j\in J\}$ and $(\bbar_j:j\in J)$ is indiscernible over $C\cup\bigcup\{\abar_i:i\in I\}$.
\end{definition}

In \cite{MonNIP}, in order to recover Theorem 1.1, it suffices to recover Proposition~3.11, so in the notation there, define 
$$(2^*) \quad \hbox{$T$ is dp-minimal and has endless indiscernible triviality}$$
which is identical to the existing (2), but now with endless indiscernible triviality replacing indiscernible-triviality.  

Again in the notation of Proposition 3.11, we must show that $(2^*)\Rightarrow (3)$ and that $(1)\Rightarrow (2^*)$.

The existing proof that $(2) \Rightarrow (3)$ is easily modified to show $(2^*)\Rightarrow (3)$. The only issue is that the convex piece $I'$ containing $\abar_i$ might not be endless. But in this case, the convex piece containing $\abar_j$ must be $I \bs I'$, which is endless. So we may conclude the argument substituting $\abar_j$ for $\abar_i$ and $I \bs I'$ for $I'$.

Establishing the implication $(1) \Ra (2^*)$ is more involved, where $(1)$ states that $T$ has the f.s.\ dichotomy. Without going through the problematic $(2)$, the paper still contains a proof of  $(1) \Ra (4)$, where $(4)$ asserts that there $T$ does not admit a pre-coding configuration.  Before tracing this proof, we recall the relevant definitions from \cite{MonNIP}.

\begin{definition} \label{def:fs dich}  $T$ has the {\em f.s.\ dichotomy} if, for all models $M$, all finite tuples $\abar,\bbar \in \C$, if $\tp(\bbar/M\abar)$ is finitely satisfied in $M$, then for any singleton $c$,
	either $\tp(\bbar/M\abar c)$ or $\tp(\bbar c/M\abar)$ is finitely satisfied in $M$.
\end{definition}

\begin{definition} \label{precoding}
	A {\em pre-coding configuration} consists of a $\phi(\xbar,\ybar,z)$ with parameters and a sequence $\I=\<\dbar_i:i\in \Q\>$, indiscernible over the parameters of $\phi$, such that for some (equivalently, for every) $s<t$ from $\Q$, there is $c\in\C$ such that
	\begin{enumerate}
		\item  $\C\models\phi(\dbar_s,\dbar_t,c)$;
		\item  $\C\models \neg\phi(\dbar_s,\dbar_v,c)$ for all $v>t$; and
		\item $\C\models \neg\phi(\dbar_u,\dbar_t,c)$ for all $u<s$.
	\end{enumerate}
\end{definition}

In \cite[\S 4]{MonNIP}, it is proved (without using \cite[Proposition 3.11]{MonNIP}) that if $T$ has the f.s. dichotomy, then $T$ does not admit coding on tuples, which is condition $(3)$ in \cite[Proposition 3.11]{MonNIP}. Thus the implication $(3) \Rightarrow (4)$ in \cite[Proposition 3.11]{MonNIP} shows that if $T$ has the f.s. dichotomy then $T$ does not admit a pre-coding configuration. (We take this opportunity to note that after the first sentence in the proof of $(3) \Rightarrow (4)$ in \cite[Proposition 3.11]{MonNIP}, the following should be inserted: ``By Ramsey and compactness, we may assume that the truth value of $\phi(\abar_i, \abar_j, c_{k, \ell})$ depends only on the order-type of $ijk\ell$.'')

Evidently, the existence of a pre-coding configuration is a statement about a certain configuration being consistent with $T$, hence one can  use Compactness to construct such configurations from many variations.  We record two variants in the following lemma.

\begin{lemma}  \label{precodingequivalents}  $T$ admits a pre-coding configuration if either of the following hold:
\begin{enumerate}
\item  There is a sequence $(\dbar_i:i\in\Z)$ (not necessarily indiscernible) and a formula $\phi(\xbar,\ybar,z)$ such that, for {\bf every} $s<0<t$ there is $h_{s,t}\in\C$ such that
\begin{itemize}
\item  $\models \phi(\dbar_s,\dbar_t,h_{s,t})$: 
\item   $\models \neg \phi(\dbar_u,\dbar_t,h_{s,t})$ for  every $u<s$;  and
\item  $\models \neg \phi(\dbar_s,\dbar_v,h_{s,,t})$ for every $v>t$.
\end{itemize}
\item Or there is an indiscernible sequence $(\dbar_i:i\in \Z)$ and a formula $\phi(\xbar,\ybar,z)$ such that, for {\bf some} $s<0<t$ there is $h_{s,t}\in\C$ such that
\begin{itemize}
\item  $\models \phi(\dbar_s,\dbar_t,h_{s,t})$: 
\item   $\models \neg \phi(\dbar_u,\dbar_t,h_{s,t})$ for  every $u<s$;  and
\item  $\models \neg \phi(\dbar_s,\dbar_v,h_{s,t})$ for every $v>t$.
\end{itemize}
\end{enumerate}
\end{lemma}

\begin{proof}  (1)  is immediate by Compactness.  For (2), we first extend our given indiscernible sequence $(\dbar_i:i\in\Z)$ to an indiscernible sequence
$(\dbar_i:i\in\Q)$, maintaining the extra conditions that $\neg\phi(\dbar_i,\dbar_t,h_{s,t})$ for all $i<s$, $i\in \Q$ and that $\neg\phi(\dbar_s,\dbar_i,h_{s,t})$ for all $i>t$, $i\in\Q$
in three steps, all using Compactness.  First, since $(\dbar_i:i<s,i\in\Z)$ is an infinite, indiscernible sequence over $(\dbar_i:i\ge s)$, for which $\neg\phi(\dbar_i,\dbar_t,h_{s,t})$
for every such $i$, by Compactness there is an extension of this segment to $(\dbar_i:i<s,i\in\Q)$ maintaining indiscernibility of the entire expanded sequence, as well
as $\neg\phi(\dbar_i,\dbar_t,h_{s,t})$.  Dually,  we can find an extension $(\dbar_i:i>t,i\in\Q)$  of $(\dbar_i:i>t, i\in\Z)$
maintaining indiscernibility with $\neg\phi(\dbar_s,\dbar_i,h_{s,t})$ for every $i>t$, $i\in\Q$.  
Finally, for the middle segment $(\dbar_i:s<i<t, i\in \Z)$,  we only need to maintain indiscernibility.  Although the sequence $(d_i:s<i<t, i\in\Z)$ is finite,
it is part of an endless indiscernible sequence.  Thus, it follows by Compactness that there is an extension  $(\dbar_i:s<i<t,i\in\Q)$ of $(\dbar_i:s<i<t, i\in\Z)$,
for which the entire sequence $(\dbar_i:i\in\Q)$ is indiscernible.
So we have constructed an indiscernible sequence $(\dbar_i:i\in\Q)$ with some distinguished  pair $s<t$ for which a witnessing element $h_{s,t}$ exists.
However, as $Aut(\Q,<)$ is 2-homogeneous and since every $\sigma\in Aut(\Q,<)$ induces an automorphism of $\C$, we conclude that for every $s'<t'$, a witnessing
element $h_{s',t'}$ exists.  Thus, we obtain a pre-coding configuration.
\end{proof}

We now assume $T$ has the f.s.\ dichotomy.   The proof that $T$ is dp-minimal in the existing proof of $(1)\Ra(2)$ in \cite[Proposition~3.11]{MonNIP} is unchanged. In fact, the proof gives the following stronger statement.
\begin{lemma} \label{dpmin+}
	If $T$ has the f.s. dichotomy, then for every indiscernible sequence $\II = (\abar_i : i \in I)$ and singleton $b$, there is a partition $I = I_0 ^\smallfrown I_1 ^\smallfrown I_2$ where $I_1$ is either empty or a singleton, such that $(\abar_i : i \in I_0)$ and $(\abar_i : i \in I_2)$ are mutually indiscernible over $b\II_1$.
	
	In the case where $I$ is Dedekind complete, we may assume $I_1$ is a singleton.
\end{lemma}
We will assume that $T$ has the f.s. dichotomy but fails endless indiscernible triviality and eventually arrive at one of the two clauses of Lemma~\ref{precodingequivalents}, giving our contradiction.
Endlessness of $(I,<)$ is crucial as once we cut the indiscernible sequence $\II$ into two mutually indiscernible halves, we
still have that each half is an infinite indiscernible sequence and thus can be extended.  In a nutshell, this extendibility of each half is what is failing in Example~\ref{example}.

\begin{lemma}  \label{Ded}  Suppose $T$ has the f.s.\ dichotomy, $(I,<)$ is an endless, Dedekind complete linear order, $\II=(\abar_i:i\in I)$ is indiscernible over $\emptyset$, but not
over $b$ for some singleton $b\in\C$.  Then there are $i^*\in I$, a finite $F\subseteq \C$, and an $F$-definable $\delta(\xbar,y)$ such that
\begin{enumerate}
\item  $(\abar_i:i\in I)$ is indiscernible over $F$;
\item  The subsequences $(\abar_i:i<i^*)$ and $(\abar_i:i>i^*)$ are mutually indiscernible over $Fb\abar_{i^*}$;
\item  The sequence of truth values of $(\delta(\abar_i,b):i\in I)$ is not constant.
\end{enumerate}
\end{lemma}

\begin{proof} 
Since $\II$ is not indiscernible over $b$, we apply Lemma \ref{dpmin+}, and let $i^*$ be the singleton element of $I_1$ there.
		Since $(I,<)$ is endless, choose $i^*_-,i^*_+\in I$ with $i^*_-< i^*< i^*_+$.
		 Choose a formula $\phi(\xbar_1,\dots,\xbar_n,b)$ witnessing that  $\II$ is not indiscernible over $b$.
		 By mutual indiscernibility over $\abar_{i^*}b$ there must be some $1\le k\le n$ for which:
		 for some/every $i_1<i_2<\dots<i_{k-1}<i^*_-$, for some/every $i^*_+<i_{k+1}<\dots<i_n$,
		 the truth values of  the three statements 
		 \begin{itemize} 
		 \item $\phi(\abar_{i_1},\dots,\abar_{i_{k-1}},\abar_{i^*_-},\abar_{i_{k+1}},\dots,\abar_{i_n},b)$;
		 \item $\phi(\abar_{i_1},\dots,\abar_{i_{k-1}},\abar_{i^*},\abar_{i_{k+1}},\dots,\abar_{i_n},b)$;
		 \item $\phi(\abar_{i_1},\dots,\abar_{i_{k-1}},\abar_{i^*_+},\abar_{i_{k+1}},\dots,\abar_{i_n},b)$
		 \end{itemize}
		 are non-constant.  Let $I':=(\ell_1,\dots,\ell_{k-1})\smallfrown I\smallfrown (r_{k+1},\dots,r_{n})$ extend $I$.
		 By Compactness, choose $n-1$ new tuples $(\abar_{\ell_1},\dots,\abar_{\ell_{k-1}})$, $(\abar_{r_{k+1}},\dots,\abar_{r_n})$
		 such that the extended sequences $(\abar_i:i<i^*, i\in I')$ and $(\abar_i:i>i^*, i\in I')$ remain mutually indiscernible over $\abar_{i^*}b$.
		 Put $$F:=\bigcup\{\abar_{\ell_i}:1\le i\le k-1\}\cup\bigcup \{\abar_{r_{i}}:k+1\le i\le n\}$$ and let $\delta(\xbar,y):=\phi(\abar_{\ell_1},\dots,\abar_{\ell_{k-1}},\xbar,
		 \abar_{r_{k+1}},\dots, \abar_{r_n},y)$.
		 This works.
		 \end{proof}
		 
		 \begin{lemma}
			Suppose $T$ has the f.s. dichotomy. Then $T$ has endless indiscernible triviality.
		 \end{lemma}
		 \begin{proof}
		 Assume, by way of contradiction, that $T$ fails endless indiscernible triviality.  An easy induction on $|B|$ gives that there is some finite $A$ and singletons $b,c$
		 for which some endless $(I,<)$ supports  a sequence $(\abar_i:i\in I)$ that is indiscernible over $Ab$ and $Ac$, but not over $Abc$.  
		 By adding constants to the language we may assume $A=\emptyset$ and, as $(\Z,<)$ embeds into any endless linear order, we may take $I=\Z$.
		 Summarizing,  we assume  the existence of a sequence $(\abar_i:i\in\Z)$ that is indiscernible over $b$ and $c$ individually, but not over $bc$.  
		 Now, working over $c$, apply Lemma~\ref{Ded} to this sequence and $b$ to obtain $i^*\in\Z$, a finite set $F$ and an $Fc$-definable $\delta(\xbar,b)$ as there.
		 To make the dependence on $c$ explicit, write $\delta$ as $\delta(\xbar,y,c)$, so $\delta(\xbar,y,z)$ is $F$-definable.  
		 As $(\Z,<)$ is transitive, we may assume $i^*=0$.  We summarize the situation from the point of view of $b$, which we now label as $b_0$.
		 We have:
		 \begin{enumerate}
		 \item  $(\abar_i:i\in\Z)$ is indiscernible over $Fc$; and 
		 \item for $b=b_0$, we have
		 \begin{enumerate}
		 \item  $(\abar_i:i\in\Z)$ is indiscernible over $Fb_0$;
		 \item  $(\abar_i:i<0)$ and $(\abar_i:i>0)$ are mutually indiscernible over $Fcb_0$;
		 \item  The truth value of $(\delta(\abar_i, ,b_0,c):i\in\Z)$ is non-constant.
		 \end{enumerate}
		 \end{enumerate}
		 
		 Because of (1), there is an automorphism $\sigma$ of $\C$ fixing $Fc$ with $\sigma(\abar_i)=\sigma(\abar_{i+1})$ for all $i\in Z$.  Let $b_j:=\sigma^j(b)$, the $j$-fold
		 iteration of $\sigma$ (this also makes sense for $j=0$ and $j<0$).  Thus, with the same $Fc$ and $(\abar_i:i\in\Z)$, we have that for every $j\in\Z$,
		 $(\abar_i:i\in Z)$ is indiscernible over $Fb_j$: $(\abar_i:i<j)$ and $(\abar_i:i>j)$ are mutually indiscernible over $Fcb_j$; and the truth value of
		 $(\delta(\abar_i,b_j,c):i\in \Z)$ is non-constant. We remark that we have again made crucial use of the endlessness of our indiscernible sequence to extend $b$ from a singleton to a whole sequence.
		 
		 Now, keeping $Fc$ fixed, we `couple' each $b_j$ by its corresponding $\abar_j$, and then by Ramsey's Theorem and Compactness we get that for any endless $(J,<)$
		 there are tuples $(\abar_j b_j:j\in J)$ (possibly distinct from the original elements) satisfying the following conditions:
		 \begin{enumerate}
		 \item The sequence  $((\abar_j b_j):j\in J)$ is indiscernible over $Fc$;
		 \item  For all $j\in J$,
		 \begin{enumerate}
		 \item  The sequence $(\abar_i:i\in J)$ is indiscernible over $Fb_j$; 
		 \item  The subsequences $(\abar_i:i<j)$ and $(\abar_i:i>j)$ are mutually indiscernible over $Fcb_j$; and
		 \item  The truth values of $(\delta(\abar_i,,b_j,c):i\in J)$ is non-constant.
		 \end{enumerate}
		 \end{enumerate}
		 
		 The following claim will allow us to define a pre-coding configuration.
		 
		 \begin{claim} 
	There is a sequence $(\dbar_r b_r:r\in\R)$ that is indiscernible over $Fc$ with $(\dbar_r:r\in\R)$ indiscernible over $Fb_0$ and an $F$-definable formula $\psi(\xbar,y,z)$  
		such that 
	\begin{enumerate}  
		\item  For all $r,s\in\R$,  $\models \psi(\dbar_r,b_s,c)$ if and only if $r=s$; and
		\item  For every singleton $c' \in \C$ and $r\in \R$, there is at most one $s\in\R$ such that $\models \psi(\dbar_s, b_r,c')$.
	\end{enumerate}
	\end{claim}

	\begin{claimproof} Consider the sequence $(\abar_jb_j:j\in J)$ obtained above with $J = 3 \times \R$.  
	As notation, for each $r\in\R$ write each `triple' as
	$(\abar_{r_-},\abar_r,\abar_{r_+})$ and let $\dbar_r:=\abar_{r_-}\abar_r\abar_{r_+}$ be the concatenation of the triple.  In what follows
	we only consider $b_r$ for each $r\in\R$.  Finally, put 
	$$\psi(\xbar_-\xbar\xbar_+,y,z):=\neg\left[\delta(\xbar_-,y,z)\leftrightarrow\delta(\xbar,y,z)\leftrightarrow\delta(\xbar_+,y,z)\right]$$
	Thus, we have $(\dbar_r b_r:r\in\R)$ is indiscernible over $Fc$, and for each $r\in \R$ we have  $(\dbar_s:s\in\R)$ is indiscernible over $Fb_r$ and
	the pair of subsequences  
	$(\dbar_s:s<r)$ and $(\dbar_s:s>r)$ are mutually indiscernible over $Fcb_r$.  Moreover,
	 for any $r,s\in\R$, $\models \psi(\dbar_s,b_r,c)$  if and only if $r=s$.  
	 
	 To get the final clause, choose any $c'\in\C$ and $r\in \R$.  We know the  original sequence $(\abar_i:i\in 3\times\R)$ is indiscernible over $Fb_r$.  If it is also
	 indiscernible over $Fb_rc'$, then the truth value of $(\delta(\abar_i,b_r,c'):i\in 3\times \R)$ is constant, hence 
	  $\models \neg\psi(\dbar_s,b_r,c')$ holds for every $s\in\R$.  On the other hand, if it fails to be indiscernible over $Fb_rc'$, then, working
	 over $Fb_r$, we apply Lemma~\ref{Ded} to the sequence $(\abar_i:i\in 3\times\R)$.  Choose $i^*\in 3\times \R$ for which
	 the subsequences $(\abar_i:i<i^*)$ and $(\abar_i:i>i^*)$ are mutually indiscernible over $Fb_rc'$.  Choose $s\in\R$ such that $i^*\in\{s_-,s,s_+\}$.
	 Then for any $t\neq s$, with $t\in\R$, the triple $(t_-,t,t_+)$ lies in one of the two subsequences.  Thus, by indiscernibility we have
	 $\models \neg\psi(\dbar_t,b_r,c')$ for all $t\neq s$.
	 \end{claimproof}

	Continuing, as $(\dbar_i:i\in\R)$ is indiscernible over  $b_0$, choose an automorphism $\sigma\in Aut(\C)$ such that $\sigma(\dbar_j)=\dbar_{j+1}$ for every $j\in\R$, and also $\sigma(b_0)=b_0$.
	For  each $i\in\Z^+$, let $\sigmai$ denote the $i$-fold composition of $\sigma$, so e.g., $\sigmai(\abar_j)=\abar_{j+i}$, while $\sigmai(b_0)=b_0$.
	As notation, put $c_i:=\sigmai(c)$.
	
	For each $m\in\Z^+$, let $B_m=\set{j\in(-\infty,0) | \models \psi(\dbar_j,b_j,c_m)}$.
	There are now two cases, each of which leads to a pre-coding configuration.
	
	\medskip\noindent
	{\bf Case 1.}  Some $B_m$ is not well ordered.
	
	\medskip
	
	Fix such an $m \in \Z^+$.  Fix a strictly decreasing sequence $J=(j_n:n\in\omega)$ from $B_m$ and put
	$I:=(i\in \Z^+ : i\ge m)$. Thus $K:=J^\smallfrown 0^\smallfrown I$ describes a subordering of $(\R,<)$ of order type $(\Z,<)$.  
	For $k\in K$, let  $\ebar_k$ denote the concatenation $\dbar_k b_k$ and let 
	$\theta(\xbar_1y_1,\xbar_2 y_2,z):=\psi(\xbar_2,y_1,z)$.  That $(\ebar_k : k \in K)$ and $\theta$ satisfy the hypotheses of Fact~\ref{precodingequivalents}(2) with $s=0$, $t=m$,
	and $h_{s,t}=c_m$
	follows from the following
	claim.  
	
\begin{claim} 
			\leavevmode
	\begin{enumerate} 
		\item  $\models \psi(\dbar_m,b_0,c_m)$;
\item $\models \neg \psi(\dbar_i,b_0,c_m)$ for all $i > m$;
\item $\models \neg \psi(\dbar_m,b_j,c_m)$ for all $j \in J$.
	\end{enumerate}
\end{claim}

	\begin{claimproof} 
		(1)  We have $\models \psi(\dbar_0,b_0,c)$, hence applying $\sigma_m$ gives $\models \psi(\sigma_m(\dbar_0),\sigma_m(b_0),\sigma_m(c))$, i.e., $\models \psi(\dbar_m,b_0,c_m)$.

		(2) We know that for any $k>0$, $\models \neg\psi(\dbar_k,b_0,c)$, so applying $\sigma_m$ yields
		$\models \neg \psi(\dbar_{k+m},b_0,c_m)$.

		(3)
		Since $j\in B_m$, we have $\models \psi(\dbar_j,b_j,c_m)$.  But then by the final clause of Claim~1,  we have $\models \neg\psi(\dbar_m,b_j,c_m)$ for $m \neq j$.
	\end{claimproof}
	
	\medskip\noindent
	{\bf Case 2.}  Not Case 1, i.e., every $B_m$ is well ordered.    
	
	\medskip
	
	In this case, for any $i\in\Z^+$, the shifted set 
	$$B_m+i=\set {r\in (-\infty,0)| r=b+i\ \hbox{for some  $b\in B_m, i\in \Z^+$}}$$ is well ordered as well.
	Since any well-ordered subset of $(-\infty,0)$ is nowhere dense, it follows by Baire category that the complement
	$$S=\set{r\in (-\infty,0) | r\not\in B_m+i\ \hbox{for every $i,m\in\Z^+$}}$$ 
	is not nowhere dense.  Thus, $S$ contains a strictly decreasing sequence $J=(j_n:n\in\omega)$, so
	 $K:=J^\smallfrown 0 ^\smallfrown \Z^+$ has order type $(\Z,<)$.  For each $k\in K$, let $\ebar_k$ denote the concatenation $\dbar_k c_k$, let $b_{i,j} := \sigmai(b_{j-i})$, and let $\theta(\xbar_1 z_1,\xbar_2 z_2,y):=\psi(\xbar_1,y,z_2)$.
	Here, we will get an instance of pre-coding via Fact~\ref{precodingequivalents}(1), as witnessed by $(\ebar_k:k\in K)$, $\theta$, and the witnesses $b_{i,j}$ for $j<0<i$ from $K$,
	once we establish the following claim.
	
	\begin{claim}  For every $j\in S$ and $i\in\Z^+$, the following hold.
				\leavevmode
	\begin{enumerate}
		\item    $\models \psi(\dbar_j, b_{i,j},c_i)$;
		\item  For all $j'\in S\setminus\set{j}$, $\models \neg\psi(\dbar_{j'}, b_{i,j},c_i)$; and
		\item  For all $\ell>i$, $\models \neg\psi(\dbar_j, b_{i,j},c_\ell)$.
	\end{enumerate}
\end{claim}
	
	\begin{claimproof}  (1) and (2).  By Claim 1(1), we have $\models \psi(\dbar_{j-i},b_{j-i},c)$ and for any $t\neq j-i$ we have $\models \neg\psi(\dbar_t,b_{j-i},c)$.
		Applying $\sigmai$ yields $\models \psi(\dbar_j,b_{i,j},c_i)$, but  $\models \neg\psi(\dbar_{j'},b_{i,j},c_i)$ for any $j'\neq j$.
		
		For (3), given $\ell>i$, put $m:=\ell-i$.  Since $j\in S$, $j-i\not\in B_m$, so $\models \neg\psi(\dbar_{j-i},b_{j-i},c_m)$.
		Then, applying $\sigmai$ (and using $c_\ell=\sigmai(c_m)$) yields $\models \neg\psi(\dbar_j,b_{i,j},c_\ell)$, as required.
	\end{claimproof}
\end{proof}

\subsection{Well-quasi-order} \label{app:wqo}
In \cite[Theorem 5.3]{MonNIP}, the proof that  $Age(T)$ is not 4-wqo is flawed. The issue is that the formula $\phi^*(x,y,z)$ is not existential, and thus neither is the formula $\exists z \phi^*(x,y,z)$ that we use to define the edges of our graphs. Since the formula is not existential, it need not be preserved by embeddings. Thus the first sentence of the last paragraph of the proof is unjustified.

	 \bibliographystyle{alpha}
	\bibliography{Bib}

\begin{bibdiv}
\begin{biblist}

\bib{AA}{article}{
      author={Adler, Hans},
      author={Adler, Isolde},
       title={Interpreting nowhere dense graph classes as a classical notion of
  model theory},
        date={2014},
     journal={European Journal of Combinatorics},
      volume={36},
       pages={322\ndash 330},
}

\bib{BS}{article}{
      author={Baldwin, John~T},
      author={Shelah, Saharon},
       title={Second-order quantifiers and the complexity of theories},
        date={1985},
     journal={Notre Dame Journal of Formal Logic},
      volume={26},
      number={3},
       pages={229\ndash 303},
}

\bib{Blum}{article}{
      author={Blumensath, Achim},
       title={Simple monadic theories and partition width},
        date={2011},
     journal={Mathematical Logic Quarterly},
      volume={57},
      number={4},
       pages={409\ndash 431},
}

\bib{TW4}{article}{
      author={Bonnet, {\'E}douard},
      author={Giocanti, Ugo},
      author={Ossona~de Mendez, Patrice},
      author={Thomass{\'e}, St{\'e}phan},
       title={Twin-width iv: low complexity matrices},
        date={2021},
     journal={arXiv preprint arXiv:2102.03117},
}

\bib{MonNIP}{article}{
      author={Braunfeld, Samuel},
      author={Laskowski, Michael~C},
       title={Characterizations of monadic {NIP}},
        date={2021},
     journal={Transactions of the American Mathematical Society, Series B},
      volume={8},
      number={30},
       pages={948\ndash 970},
}

\bib{CPS}{article}{
      author={Cameron, Peter},
      author={Prellberg, Thomas},
      author={Stark, Dudley},
       title={Asymptotics for incidence matrix classes},
        date={2006},
     journal={The Electronic Journal of Combinatorics},
      volume={13},
      number={1},
}

\bib{DGL}{article}{
      author={Dolich, Alfred},
      author={Goodrick, John},
      author={Lippel, David},
       title={Dp-minimality: basic facts and examples},
        date={2011},
     journal={Notre Dame Journal of Formal Logic},
      volume={52},
      number={3},
       pages={267\ndash 288},
}

\bib{Gur}{article}{
      author={Gurski, Frank},
       title={Linear layouts measuring neighbourhoods in graphs},
        date={2006},
     journal={Discrete Mathematics},
      volume={306},
      number={15},
       pages={1637\ndash 1650},
}

\bib{Rap}{article}{
      author={Macpherson, HD},
       title={Infinite permutation groups of rapid growth},
        date={1987},
     journal={Journal of the London Mathematical Society},
      volume={2},
      number={2},
       pages={276\ndash 286},
}

\bib{MacHom}{article}{
      author={Macpherson, HD},
       title={A survey of homogeneous structures},
        date={2011},
     journal={Discrete Mathematics},
      volume={311},
      number={15},
       pages={1599\ndash 1634},
}

\bib{Morley}{article}{
      author={Morley, Michael},
       title={Categoricity in power},
        date={1965},
     journal={Transactions of the American Mathematical Society},
      volume={114},
      number={2},
       pages={514\ndash 538},
}

\bib{parigot1982theories}{article}{
      author={Parigot, Michel},
       title={Th{\'e}ories d'arbres},
        date={1982},
     journal={The Journal of Symbolic Logic},
      volume={47},
      number={4},
       pages={841\ndash 853},
}

\bib{Hanf}{article}{
      author={Shelah, Saharon},
       title={Monadic logic: {H}anf numbers},
        date={1986},
       pages={203\ndash 223},
        note={Sh:197, available at \url{shelah.logic.at/files/95757/197.pdf}},
}

\bib{Shc}{book}{
      author={Shelah, Saharon},
       title={Classification theory: and the number of non-isomorphic models},
   publisher={Elsevier},
        date={1990},
}

\bib{Pierre}{book}{
      author={Simon, Pierre},
       title={A {G}uide to {NIP} {T}heories},
   publisher={Cambridge University Press},
        date={2015},
}

\bib{Sim21}{article}{
      author={Simon, Pierre},
       title={A note on stability and {NIP} in one variable},
        date={2021},
     journal={arXiv preprint arXiv:2103.15799},
}

\bib{ST}{article}{
      author={Simon, Pierre},
      author={Toru{\'n}czyk, Szymon},
       title={Ordered graphs of bounded twin-width},
        date={2021},
     journal={arXiv preprint arXiv:2102.06881},
}

\end{biblist}
\end{bibdiv}

\end{document}